\documentclass[a4paper,twoside,notitlepage]{article}
\usepackage{amsthm,amsmath,amscd}
\usepackage{amssymb}
\usepackage{graphicx}
\usepackage{psfrag}
\usepackage{latexsym, epsfig}
\usepackage[dvipsnames]{xcolor}
\usepackage{hyperref}
\usepackage[bbgreekl]{mathbbol}

\usepackage{tikz}
\usepackage{fp}

\newfont{\smc}{cmcsc10 at 10pt}
\newfont{\ind}{cmcsc10}
\newfont{\testo}{cmr10 at 10pt}
\newfont{\mail}{cmtt10}
\newfont{\tit}{cmbx10 at 18pt}

\usepackage[margin=1.8cm]{geometry}

\newtheorem{lemma}{Lemma}
\newtheorem{theorem}[lemma]{Theorem}
 
\newtheorem{proposition}[lemma]{Proposition} 
\newtheorem{definition}[lemma]{Definition} 
\newtheorem{remark}[lemma]{Remark}

\newcommand{\rin}{\rm in}
\newcommand{\out}{\rm out}
\newcommand{\norm}[1]{\lVert #1 \rVert}

\newcommand{\dH}{ \delta_{H}}
\newcommand{\dzero}{ \delta_{0}}

\newcommand{\dss}{{\delta}}
\newcommand{\dssb}{\overline{\delta}}

\newcommand{\bel}[1]{\begin{equation}\label{#1}}
\def\bas#1\eas
{\begin{align*}#1\end{align*}}
\newcommand{\eal}{\end{align} }
\def\ba#1\ea
{\begin{align}#1\end{align}}
\newcounter{stepnb}
\newcommand{\firststep}{\setcounter{stepnb}{0}}
\newcommand{\step}[1]{{{\sc \addtocounter{stepnb}{1}\vskip.35\baselineskip\noindent $\circleddash$ Step \arabic{stepnb}:} #1.}} 
\newcommand{\vSC}[2]{{ P_{#1,#2}}}
\newcommand{\vSB}[2]{{ S_{#1,#2}}}
\newcommand{\ftS}[2]{{ S^{FT}_{#1,#2}}}
\newcommand{\ww}{{w}}
\newcommand{\uu}{{ u}}
\newcommand{\vv}{{ v}}

\newcommand{\be}{\begin{equation}}
\newcommand{\ee}{\end{equation}}

\newcommand{\conv}{\mathop {\rm conv}\nolimits}

\newcommand{\conc}{\mathop {\rm conc}\nolimits}
\newcommand{\rarbound}{2\eps}
\newcommand{\elleuno}{\mathbb{L}^{\strut 1}}

\newcommand{\real}{\mathbb{R}}
\newcommand{\R}{\mathbb{R}}
\newcommand{\nat}{\mathbb{N}}
\newcommand{\N}{\mathbb{N}}

\newcommand{\Q}{\mathcal{Q}}

\newcommand{\MC}{\mathcal{M}}
\newcommand{\U}{\mathcal{U}}
\newcommand{\Ci}{\mathcal{C}}

\newcommand{\Z}{\mathbb{Z}}
\newcommand{\OL}{\mathcal{O}}

\newcommand{\I}{\mathcal{I}}
\newcommand{\J}{\mathcal{J}}
\newcommand{\E}{\mathcal{E}}

\newcommand{\NP}{\mathcal{NP}}
\newcommand{\br}{\mathbf{r}}

\newcommand{\eps}{\varepsilon}
\newcommand{\dds}{\frac{\rm d}{{\rm d}s}}

\newcommand{\TV}{\hbox{\testo Tot.Var.}}
\newcommand{\oV}{\overline{V}}
\newcommand{\oT}{\overline{T}}

\newcommand{\SBV}{\mathrm{SBV}}
\newcommand{\BV}{\mathrm{BV}}

\newcommand{\dtau}{\displaystyle{\frac{d}{d\tau}}}
\newcommand{\second}{{\prime\prime}}

\newcommand{\scalar}[2]{{\big\langle #1, #2 \big\rangle}}

\DeclareMathOperator{\Graph}{Graph}

\begin{document}

%
%
\title{{\huge On the Structure of BV Entropy Solutions}\\
\vspace{.1truecm}
{\huge  for  Hyperbolic Systems of Balance Laws}
\\
\vspace{.1truecm}
{\huge  with General Flux Function}
      \vspace{0.1in}
}

\author{
    {\scshape Fabio Ancona, Laura Caravenna, Andrea Marson}
    \thanks{
    Dipartimento di Matematica `Tullio Levi-Civita'
      Via Trieste, 63,
      35121 - Padova, Italy;                       
      \newline 
      e-mail: \texttt{ancona@math.unipd.it, laura.caravenna@unipd.it,
      marson@math.unipd.it}}     \\\today
      }

\date{}
\maketitle
\thispagestyle{empty}
\vspace{-0.10cm}
\begin{abstract}
\noindent
The paper describes the qualitative structure of BV entropy solutions of a general strictly hyperbolic system of balance laws 
with characteristic fields either piecewise genuinely nonlinear or linearly degenerate. In particular, we provide an accurate
description of the local and global wave-front structure of
a BV solution generated by a fractional step scheme combined with a wave-front tracking algorithm.
This extends the corresponding results
in~\cite{BYu} for  strictly hyperbolic systems of conservation laws.
\end{abstract}
\vspace{0.5cm}

2010\textit{\ Mathematical Subject Classification:} 35L45, 35L65


\thispagestyle{plain}
%
%
%
%

\tableofcontents
\section[Introduction]{Introduction}
\label{sec:int}
\indent

Consider the Cauchy problem for a general hyperbolic system
of $N$ quasilinear first order PDEs in one space dimension
\begin{align}
\label{eq:sysnc}
& u_t+A(u) u_x =g(t,x,u)\,,\\
\noalign{\smallskip}
\label{eq:inda}
& u(0,x) =\overline u(x)\,.
\end{align}
Here the vector $u=u(t,x)=\big(u_1(t,x),\dots, u_N(t,x)\big)$, and
$A=A(u)$ is a smooth matrix-valued function defined on a domain
$\Omega\subseteq\real^N$. Solution to~\eqref{eq:sysnc}-\eqref{eq:inda}
are considered as limits in $L^1_{loc}$ of vanishing viscosity approximations
\be
\label{eq:vv}
u^\eps_t+A(u^\eps) u^\eps_x =g(t,x,u^\eps) + \eps u^\eps_{xx}
\ee
as $\eps\to 0+$. In case $A$ is the jacobian matrix
of a \emph{flux function} $F:\Omega\rightarrow \real^N$, then \eqref{eq:sysnc}
can be written as a system of balance laws, namely
\begin{equation}
\label{eq:sysc}
u_t+F(u)_x =g(t,x,u)\,.
\end{equation}
We assume that
system \eqref{eq:sysnc} is strictly hyperbolic, i.e. that the
matrix $A(u)$ has $N$ real distinct eigenvalues
\be
\lambda_1(u)<\dots<\lambda_N(u)\qquad\forall~u\,,
\label{eq:strhyp}
\ee
and we will denote by
\be
\label{eq:rle}
r_1(u),\ldots, r_N(u)\,, \qquad
l_1(u),\ldots,l_N(u)
\ee
corresponding bases of, respectively, right and left
eigenvectors, normalized so that
\be
\label{eq:norm}
\vert r_k(u)\vert\equiv 1\,,
\qquad
\scalar{r_k(u)}{l_h(u)} = \delta_{kh}\,,
\ee
where $\scalar{\cdot}{\cdot}$ stand for the usual scalar product in $\R^N$,
and $\delta_{hk}$ is the usual Kronecker symbol.
The limits of vanishing viscosity approximations to~\eqref{eq:sysc}-\eqref{eq:inda}
turns out to be distributional solutions which are entropy admissible (see \cite{Ch1}).
We consider the function $g:\real\times\Omega\rightarrow \real^N$
to be continuosly differentiable in the
$u$ variable and measurable in the $x$ one. Moreover, we assume
that
\begin{description}
\item[(G)]\label{Ass:G}
the function $g$ in~\eqref{eq:sysnc} is continuous in $t$ and Lipschitz
continuous w.r.t. $x$ and $u$, uniformly in $t$; moreover, there exists
function $\alpha\in L^1(\R)$ such that $\vert g_x(t,x,u)\vert \leq
\alpha(x)$\label{alpha} for any $t,u$.
\end{description}
Regarding the assumptions
on $g$, since the seminal papers~\cite{DafHsiao, TPLin} and
the first paper~\cite{CP} on the well-posedness
of the Cauchy problem,
many papers appeared in the past years dealing with several
existence results for first order hyperbolic
inhomogeneous systems, both local and global in time, provided the initial
datum $\overline u$ has suitably small total variation.
See Subsection~\ref{subsec:few} below. Moreover, as long as
one is interested in a local (both in time and space) existence result,
the assumption $\vert g_x(t,x,u)\vert \leq \alpha(x)$ with $\alpha\in L^1(\R)$
is not relevant.
\vspace{.1truecm}

Aim of this paper is to provide some preliminaries that will be used
to prove that, for a.e. time $t$ in the interval of existence of a solution $u=u(t,x)$
to~\eqref{eq:sysnc}-\eqref{eq:inda}, $u(t,\cdot)$ enjoys a SBV regularity, i.e.
it is BV and $\partial_x u(t,\cdot)$ does not have a Cantor part. In order
to pursue this result, in the present paper a couple of achievements are presented.
\begin{enumerate}
\item
We briefly introduce a method to construct a piecewise constant
approximate solution for a non conservative
system~\eqref{eq:sysnc}, under the only regularity and strict hyperbolicity
assumptions on the matrix $A$. The algorithm we introduce, that will lead
to an existence result for~\eqref{eq:sysnc}-\eqref{eq:inda} at least
locally in time, follows the
guidelines contained in~\cite{AGG} and~\cite{AG2} for inhomogeneous
systems, and in~\cite{AMfr} for systems without assumptions of genuine
nonlinearity or linear degeneracy on the characteristic fields.
\begin{theorem}
\label{Th:Exloc}
Consider a strictly hyperbolic system of balance laws
\eqref{eq:sysnc}-\eqref{eq:inda}-\eqref{eq:strhyp} and suppose
Assumption (G) at Page~\pageref{Ass:G} holds.
Then there exists $\dssb, T>0$ and a sub-domain $\mathfrak D$ of
$\{\overline v\in L ^{1}(\R;\R^{N})\cap \BV(\R;\R^{N})\ :\ \TV(\overline v)<\dssb\}$
of initial data such that the following hold.
Suppose $u$ is the vanishing viscosity solution~\cite{Ch1} of the
Cauchy problem~\eqref{eq:sysnc}-\eqref{eq:inda} with a given initial datum
$\overline{\uu}\in\mathfrak D $.
We construct a piecewise-constant, fractional step approximation $u_\nu$
of the vanishing viscosity solution $u$ such that $u_\nu(t,\cdot)$
converges to $u(t,\cdot)$ in $L^1(\real)$ for $0\leq t<T$.
\end{theorem}
The precise statement is given in Theorems~\ref{T:localConv}-\ref{T:localEst} below.
We stress that
\begin{itemize}
\item
Our local in time result yields a piecewise constant approximation for two interesting physical models,
see \S~\ref{Ss:convergence}.
This presently was not available.
\item
It is interesting to combine our local-in-time Theorem~\ref{Th:Exloc}, or a vanishing-viscosity local-in-time version of the existence theorem in~\cite{Ch1}, with the works by Dafermos~\cite{Dafdiss1, Dafdiss2, Dafrelax1, Dafrelax2} which allow to pass from local-in-time to global-in-time existence.
\end{itemize}

\item
We define some measures related to the approximate solutions that
eventually will converge weakly
to $\partial_x u(t,\cdot)$. They will be fundamental
for describing the qualitative structure of BV entropy solutions of a general strictly hyperbolic system of balance laws 
with characteristic fields either piecewise genuinely nonlinear or linearly degenerate. 
Details are available in \S~\ref{Ss:structure}
and Theorem~\ref{Th:structure} below, extending the works~\cite{BreLF,BressanBook,BYu} relative to the homogeneous system.
\end{enumerate}

\subsection[Few results available on balance laws]{Few results available on balance laws}
\label{subsec:few}

In the following, we sumarize
a few results, and we refer to the original papers and the references therein
for a complete treatment.
\begin{itemize}
\item
In~\cite{AG2} the authors consider a system in conservation form~\eqref{eq:sysc}
with each chracteristic field $r_k$ genuinely nonlinear or linearly degenerate
in the sense of Lax~\cite{Lax}, and the source term $g$ is assumed
to depend only on $u$, and not on $(t,x)$. They prove a couple of results.
\begin{itemize}
\item
A local existence theorem with the only assumptions that $g\in\Ci^2$  and
$g(0)=0$.
\item
A global existence theorem assuming, besides the above assumptions on $g$,
that the system is \emph{diagonally
dominant}, i.e., denoting by $R(u)$ the matrix whose columns are the vectors
$r_k$, and letting
\be
\label{eq:matrixG}
G\doteq R^{-1}(0) \cdot D_u g(0) \cdot R(0) = (G_{ij})_{i,j=1,\ldots,N}\,,
\ee
the entries $G_{i,j}$ of $G$ satisfy
\be
\label{eq:dd1}
G_{ii}+\sum_{\substack{j=1\\ j\neq i}}^n \vert G_{ij}\vert \leq
-c\qquad \forall i=1,\ldots,n\,,
\ee
for some positive constant $c$.
Moreover, the authors provide a uniqueness result.
\end{itemize}
\item
\cite{AGG} deals again with a system in conservation form with genuinely nonlinear
and linearly degenerate characteristic fields. Moreover, the authors assume the system
to be \emph{non resonant}. i.e. that all the eigenvalues of $A(u)=DF(u)$ be
bounded away from zero, i.e.
\be
\label{eq:nrc}
\vert \lambda_i(u)\vert \geq c>0\qquad
\forall i=1,\ldots,N\,,
\ee
for some $c>0$. Regarding the source term, it does not depend on $t$, so
that $g=g(x,u)$, and, besides
the regularity assumptions ($\Ci^2$ in $u$ and measurable in $x$), the authors
require the existence of a bounded $L^1$ function $\omega=\omega(x)$
such that
\be
\label{eq:diss1}
\vert g(x,u) \vert + \vert D_ug(x,u) \vert \leq \omega(x)
\qquad \forall x\in \R\,, ~\forall u\in\R^N\,.
\ee
Under this assumptions, a global solution to a Cauchy problem for~\eqref{eq:sysc}
is constructed by means of suitable front tracking approximations. Such a solution
is proved to be unique.
\item
In~\cite{Ch1} a global solution to~\eqref{eq:sysnc}-\eqref{eq:inda} is constructed
by means of vanishing viscosity approximations~\eqref{eq:vv}, following the
approach contained in~\cite{BB} for homogeneous systems. No assumptions
are considered on the matrix $A=A(u)$, but the regularity and the strict hyperbolicity.
Instead, the source term $g$ is share the same assumptions in~\cite{AG2},
and hence it enjoys a $\Ci^2$ regularity and it is diagonally dominant in the sense
we explained above. In~\cite{Ch2}
the same author prove that the solution to a Cauchy problem as the one
obtained in~\cite{Ch1} is actually unique.
\item
In~\cite{Dafdiss1} and~\cite{Dafdiss2} two kinds of dissipative assumptions
are considered, to weaken the one contained in~\cite{AG2}.
\begin{itemize}
\item
In~\cite{Dafdiss1}, instead of condition~\eqref{eq:dd1}, it is simply required
that $G_{ii}>0$ for any $i=1,\ldots,N$. In order to obtain a weak entropy admissible
solution to \eqref{eq:sysc}-\eqref{eq:inda} globally defined in time, the authors
assumes that a local solution exists fulfilling the estimate
$$
\int_{\R} u(t,x)\, dx\leq b\int_{\R} \overline{u}(x)\, dx\,,
$$
for some $b>0$ and for any $t\in [0,T[$ where the solution is defined.
\item
In~\cite{Dafdiss2} system~\eqref{eq:sysc} is assumed to be endowed with
a convex antrpy-entropy flux pair $(\eta,q)$, and the matrix
$D^2\eta(0) D_ug(0)$ is required to be positive definite. Then a global solution
to \eqref{eq:sysc}-\eqref{eq:inda} does exists, provided that the intial datum
satisfies
\be
\label{eq:diss2}
\int_{\R} (1+\vert x\vert)^{2s} \vert \overline{u}(x) \vert^2 \,dx < \dss
\ee
for some $s>1$ and a sufficiently small $\dss>0$. Here, a local solution
does exist due to~\cite[Theorem~1]{DafHsiao}, where a random choice
method~\cite{Glimm} is used.
\end{itemize}
We stress that in both~\cite{Dafdiss1} and~\cite{Dafdiss2} no assumption
of genuine nonlinearity or linear degeneracy are made, while the source
term $g$ is assumed not to depend on $x$.
\item
In~\cite{Dafrelax1} and in the survey~\cite{Dafrelax2} the author deals with
BV solution of inhomegenous systems endowed with convex entropy-entropy
pair $(\eta,q)$ satisfying a dissipative condition, and satisfying the
Shizuta-Kawashima condition~\cite{ShiKaw}
\be
\label{eq:shiKaw}
D_u g(0) \, \br_i(0) \neq \mathbf{0}\quad
\forall i=1,\ldots,n\,,
\ee
where $g$ is independent on $t$ and $x$, and $g(0)=0$. Namely, the theorem
is the following
\begin{theorem}[\cite{Dafrelax1,Dafrelax2}]
\label{thm:Daf}
Consider system~\eqref{eq:sysc} with $g$ independent on $(t,x)$, and
assume it is in the form
$$
\begin{cases}
V_t+K(V,W)_x=0\\
\noalign{\smallskip}
W_t+H(V,W)_x=C(V,W)W\,,
\end{cases}
$$
where $u=(V,W)$. Assume that there exists $V_0\in\R$ such that
the $V$-component of the initial datum $\overline{u} =
(\overline{V},\overline{W})$ satisfies
\be
\label{eq:hypinda1}
\int_{\R} \big( \overline{V}(x)-V_0 \big)\, dx = 0\,.
\ee
Moreover, assume that there exist a convex entropy-entropy
pair $(\eta,q)$ and  positive constant $a$
such that the dissipative condition
\be
\label{eq:diss}
D\eta(u) \big[ g(u) - g(V_0,0) \big] \leq
-a \big\vert g(u) - g(V_0,0) \big\vert^2
\ee
holds. Then, there exist $\dzero,\sigma_0,\gamma,b,c_0,c_1>0$ such that, if
\be
\label{eq:hypinda2}
\TV \overline{u}=\dss<\dzero
\quad\text{and}\quad
\int_\R (1+x^2) \big( \big\vert \overline{V}(x)-V_0 \big\vert^2 +
\big\vert \overline{W}(x) \big\vert^2 \big) \, dx =\sigma^2<\sigma_0^2\,,
\ee
then the Cauchy problem \eqref{eq:sysc}-\eqref{eq:inda} has an
admissible $BV$ solution $u=u(t,x)$ on $[0,+\infty[\times\R$, and
\begin{gather}
\label{eq:tvest}
\int_\R \big( \big\vert \overline{V}(x)-V_0 \big\vert +
\big\vert \overline{W}(x) \big\vert \big)\, dx \leq b\sigma\,,
\quad
\TV u(t,\cdot)\leq c_0\sigma+c_1\dss e^{-\gamma t}
\qquad \forall t\geq 0\,,\\
\noalign{\smallskip}
\int_\R \big( \big\vert \overline{V}(x)-V_0 \big\vert +
\big\vert \overline{W}(x) \big\vert \big)\, dx\to 0\,,
\quad
\TV u(t,\cdot)\to 0\qquad \text{as}\quad t\to +\infty\,.
\end{gather}
\end{theorem}
Again, in~\cite{Dafrelax1,Dafrelax2} local in time solutions are provided
thanks to~\cite[Theorem~1]{DafHsiao}.
\end{itemize}

\subsection{Motivating models and global in time solutions}
\label{Ss:convergence}

Systems that do not satisfy the classical assumptions of genuine nonlinearity
or linear degeneracy in the sense of Lax~\cite{Lax} may arise
in several context. A first example is a
system of balance laws arising in modelling elasticity,
\begin{equation}
\label{eq:elas}
\begin{cases}
v_t-u_x=0\\
\noalign{\smallskip}
u_t-\sigma(v)_x = -\alpha u\,,
\end{cases}
\end{equation}
where the stress $\sigma=\sigma(v)$ satisfies $\sigma'(v)>0$.
Such a system has been diffusively studied (e.g., see~\cite{Daffricdam,DiPerna}). 
Its behavior resembles the $p$-system with dumping~\cite{Daffricdam,DafPan},
but its characteristic fields are not genuinely nonlinear, nor linearly degenerate.
Indeed, the derivatives
of the eigenvalues $\lambda_{1,2} (u,v) = \pm\sqrt{\sigma'(v)}$
of the jacobian matrix of the flux function $f(u,v) = (-u, -\sigma(v))$ along
the corresponding right eigenvectors vanish whenever $\sigma''(v)=0$.
It can be easily seen that the Shizuta-Kawashima condition~\eqref{eq:shiKaw} is fulfilled,
and that
$$
\eta(v,u) = \int \sigma(v)\, dv + \frac{1}{2}u^2
$$
is an entropy satisfying the dissipative condition~\eqref{eq:diss},
where $g(v,u)=(0,-\alpha u)$. Hence, once a local in time solution
to a Cauchy problem for~\eqref{eq:elas} has been provided, it can be
prolonged to a global in time one by
using Theorem~\ref{thm:Daf}.
\vspace{.1truecm}

Another $2\times 2$ system of balance laws not fulfilling the classical
Lax assumptions on characteristic fields is the generalized
Cattaneo's model of heat conduction in high purity
crystals~\cite{RMScimento,RMS,SRM},
\begin{equation}
\label{eq:cattaneo}
\begin{cases}
\rho e_ t +q_x=0\\
\noalign{\smallskip}
(\alpha q)_t + \nu_x  = -\dfrac{\nu'}{k} q\,.
\end{cases}
\end{equation}
In order to check that the assumptions of~Theorem~\ref{thm:Daf}
are fulfilled, we rewrite~\eqref{eq:cattaneo} in terms of the conserved
quantities $e$ and $Q=\alpha q$:
\begin{equation}
\label{eq:cattaneo2}
\begin{cases}
e_ t +\left( \dfrac{Q}{\rho\alpha} \right)_x=0\\
\noalign{\smallskip}
Q_t + \nu_x  = -\dfrac{\nu'}{\alpha k}Q\,.
\end{cases}
\end{equation}
Here,
letting $\vartheta$ be the absolute temperature, we denote with
$q=q(t,x)$ the heat flux, $e=e(\vartheta)$ the internal energy,
$\rho$ the density, which is assumed to be constant, $k=k(\vartheta)$
the heat conductivity,
while $\alpha=\alpha(\vartheta)$ and $\nu=\nu(\vartheta)$ are
constitutive functions. Regarding $e$, $\alpha$, $\nu$, up to a rescaling
of the absolute temperature $\vartheta$, we assume that~\cite{RMScimento}
\be
\label{eq:constrel1}
e(\vartheta) = \vartheta^4\,,\qquad
\alpha(\vartheta) = \dfrac{1}{\rho\vartheta U(\vartheta) \sqrt{e'(\vartheta)}}\,,\qquad
\nu'(\vartheta) = \dfrac{U(\vartheta)}{\vartheta}\sqrt{e'(\vartheta)}\,,
\ee
where $U(\vartheta)>0$ is the so called ``second sound velocity'', i.e.
the velocity of small perturbations propagating into an equilibrium state.
In the case studies contained in~\cite{RMScimento,SRM} $U$ takes the form
\be
\label{eq:constrel2}
U(\vartheta) = \dfrac{1}{\sqrt{A+B\vartheta^n}}\,,
\ee
where the positive constants $A$ and $B$ and the exponent $n$ depend on the
material under observation. Moreover, following~\cite[(21)]{Coleman},
for the heat conductivity $k=k(\vartheta)$ we can deduce the following expression
\be
\label{eq:constrel3}
k(\vartheta) = \sqrt{\dfrac{e'(\vartheta)\nu'(\vartheta)}{\alpha(\vartheta)}}
U(\vartheta)\,.
\ee
A direct computation shows that~\eqref{eq:cattaneo} is piecewise genuinely
nonlinear in the sense of Definition~\ref{D:pwGN} below (see~\cite{RMS,SRM}).
Unfortunately~\eqref{eq:cattaneo} is only weakly diagonally dominant around an
equilibrium state $(\bar{\vartheta},0)$, i.e. the entries of the corresponding matrix
$G$ at~\eqref{eq:matrixG} do not satisfy~\eqref{eq:dd1}, but
$$
G_{ii}<0\,,\qquad G_{ii} + \sum_{j\neq i} \vert G_{i,j}\vert = 0\,.
$$
Theorem~\ref{T:localConv} below states the convergence of the algorithm
described in \S\ref{sec:PCA} to a local in time solution
to a Cauchy problem~\eqref{eq:sysnc}-\eqref{eq:inda}, which a Cauchy
problem for~\eqref{eq:cattaneo} is a particular case of.
In order to extend a local solution to a global one, it can be used
the method described in \cite{Dafrelax1,Dafrelax2}.
Now we need to check that the Shizuta-Kawashima condition~\eqref{eq:shiKaw} and
the entropy dissipation condition~\eqref{eq:diss} are fulfilled.
\begin{itemize}
\item
\emph{Shizuta-Kawashima condition.}
A basis of right eigenvectors of the jacobian matrix of the flux
function
$$
F(e, Q) = \big( Q/\rho\alpha, \, 
\nu \big)
$$
is given by
$$
\begin{aligned}
\br_1(e,Q) &= \left(
-\dfrac{2\sqrt{\vartheta^3\alpha'^2 Q^2+
\rho\alpha^3\nu'}+\alpha' Q}{8\rho\vartheta^3\alpha^2\nu'},
\, 1
\right)\,,\\
\noalign{\medskip}
\br_2(e,Q) &=\left(
\dfrac{2\sqrt{\vartheta^3\alpha'^2 Q^2+
\rho\alpha^3\nu'}-\alpha' Q}{8\rho\vartheta^3\alpha^2\nu'},
\, 1
\right)\,.
\end{aligned}
$$
Let $(e,Q)=(\bar{e},0)$ be an equilibrium of`\eqref{eq:cattaneo2}, corresponding
to a constant temperature $\bar{\vartheta} = \sqrt[4]{\bar{e}}$.
The Shizuta-Kawashima condition~\eqref{eq:shiKaw}
$$
D_{e,Q} (0,-\nu'Q/\alpha k)\bigg\vert_{(e,Q)=(\bar{e},0)}
\br_i(\bar{e},0)\neq \mathbf{0}\quad i=1,2\,,
$$
reduces to
$$
\dfrac{\nu'(\bar{\vartheta})}{\alpha(\bar{\vartheta}) k(\bar{\vartheta})} \neq 0\,.
$$
\item
\emph{Entropy dissipation condition.} An entropy for system~\eqref{eq:cattaneo2}
is given by (see~\cite[(2.11)-(2.12)]{RMScimento})
$$
\eta(e,Q) = -\dfrac{4}{3}\rho e^{3/4} + \dfrac{1}{2\gamma}Q^2\,.
$$
Hence, the entropy dissipation condition~\cite[(2.7)]{Dafrelax2} is written
$$
D_{e,Q}\eta(e,Q) \cdot (0,-\nu'Q/\alpha k) \leq -a \dfrac{\nu'^2}{\alpha^2 k^2} Q^2
$$
for some $a>0$, and $(e,Q)$ in a neighbourhood of the equilibrium point
$(\bar{e},0)$, so that
\be
\label{eq:entrdiss}
\dfrac{1}{\gamma}\geq a \dfrac{\nu'}{\alpha k}
\ee
must holds in a neighbourhood of $(\bar{e},0)$. By using the constitutive
relationships \eqref{eq:constrel1}-\eqref{eq:constrel3}, we get
that~\eqref{eq:entrdiss} holds true once $a\leq 1/\gamma\sqrt{\rho}$.
\end{itemize}
It follows that we can apply~Theorem~\ref{thm:Daf}, and, if an intial datum
for~\eqref{eq:cattaneo2} fulfills assumptions~\eqref{eq:hypinda1}
and~\eqref{eq:hypinda2}, then a global in time solution to~\eqref{eq:cattaneo2}
exists, and it satisfies~\eqref{eq:tvest}. In particular, using the construction
of \S\S~2 and~3, we can provide a piecewise constant approximate solution
to~\eqref{eq:cattaneo2} globally defined in time.

\subsection{Statement on the structure of solutions to balance laws}
\label{Ss:structure}
We describe in this paper the global local structure of solutions of balance laws whose characteristic fields are either piecewise-genuinely nonlinear or linearly degenerate, for Lipschitz continuous sources $g=g(u)$. The theorem extends the works~\cite{BreLF,BressanBook,BYu} relative to the homogeneous system.
One of the expected application is to extend to this setting the $\SBV$ and $\SBV$-like regularity of solutions in a forthcoming paper.

\begin{definition}
\label{D:LD}
The $i$th-characteristic field is \emph{linearly degenerate} if\quad $\nabla\lambda_{i}(u)\cdot r_{i}(u)\equiv0$.
\end{definition}

\begin{definition}
\label{D:pwGN}
The $i$th-characteristic field is \emph{piecewise genuinely nonlinear} if the set 
\[
Z_{i}:=\{u\quad:\quad \nabla\lambda_{i}(u)\cdot r_{i}(u)=0\}
\]
is the union of $(N-1)$-dimensional disjoint manifolds $Z_{i}^{j}$, for $j=1,\dots, J_{i}$, which are transversal to the field $r_{i}(u)$ and such that each $i$-rarefaction curve $R_{i}[u_{0}]$ crosses all the $Z_{i}^{j}$.
\end{definition}

Let $S_{i}[u^{-}](s)$ denote the \textbf{$i$-th Hugoniot curve} issuing from $u^{-}$; we denote by $\sigma_{i}[u^{-}](s)$ the corresponding Rankine-Hugoniot \textbf{speed} of the $i$-th discontinuity $[u^{-},S_{i}[u^{-}](s)]$: $\sigma_{i}$ and $S_{i}$ are defined by by the Implicit Function Theorem by the relation
\[
\sigma_{i}[u^{-}](s)\, \left(S_{i}[u^{-}](s)-u^{-}\right) = f\left(S_{i}[u^{-}](s)\right)-f(u^{-})
\]
together with
\[
S_{i}[u^{-}](0)=u^{-}\ ,
\qquad
\sigma_{i}[u^{-}](0)=\lambda_{i}(u^{-})\ ,
\qquad
\dds S_{i}[u^{-}](0)=r_{i}(0)
\ .
\]
One can suppose that $ S_{i}[u^{-}]$ is parameterized by the $i$-th component relative to the basis $r_1(u),\ldots, r_N(u)$.
If $u^{+}=S_{i}[u^{-}](s)$, we denote also by $\sigma_{i}(u^{-},u^{+})=\sigma_{i}[u^{-}](s)$ the \textbf{speed} of the $i$-th discontinuity $[u^{-},u^{+}]$. This $i$-th discontinuity is \textbf{admissible} when~\cite{tplrpnpn}
\[
\forall 0\leq |\tau|\leq |s| 
\qquad
\sigma_{i}[u^{-}](\tau)\geq\sigma_{i}(u^{-},u^{+})\ .
\]
Finally, if the $i$-th field is piecewise genuinely nonlinear we call~\cite{TPLsimple} that an admissible $i$-jump $[u^{-},u^{+}]$ is called \textbf{simple} if
\[
\forall 0< |\tau|< |s| 
\qquad
\sigma_{i}[u^{-}](\tau)>\sigma_{i}(u^{-},u^{+})
\qquad
u^{+} = S_{i}[u^{-}](s)
\ .
\]
If the admissible jump $[u^{-},u^{+}]$ is not simple, we call it a \textbf{composition of the waves} $[u_{0},u_{1}]$, $[u_{1},u_{2}]$, \dots, $[u_{\ell},u_{\ell+1}]$ if
\[
u_{0}=u^{-}\ , 
\qquad
u_{\ell+1}=u^{+}\ ,
\qquad
u_{k}=S_{i}[u^{-}](s_{k})\ ,
\qquad
\sigma_{i}[u^{-}](s_{k})=\sigma_{i}(u^{-},u^{+})
\]
for all $k=0,\dots,\ell+1$ and for
\[
0=s_{0}<s_{1}<\cdots<s_{\ell}<s_{\ell+1}=s
\quad\text{or}\quad
s=s_{\ell+1}<s_{\ell}<\cdots<s_{1}<s_{0}=0\ .
\]
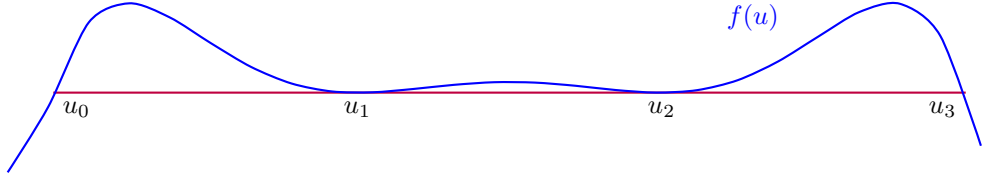
\begin{figure}\centering
\begin{tikzpicture}[xscale=4]
      \draw  (0,0) node[below right] {$u_{0}$};
      \draw  (1,0) node[below ] {$u_{1}$};
      \draw  (2,0) node[below ] {$u_{2}$};
      \draw  (3,0) node[below left] {$u_{3}$};
      \draw[blue]  (2.3,1) node[] {$f(u)$};
      \draw[purple,thick] (0,0) -- (3,0) node[below] {};
      \draw[domain=-0.15:3.05,smooth,variable=\x,thick,blue] plot ({\x},{-\x^6 +9*\x^5 - 31 *\x^4+51*\x^3-40*\x^2+12*\x});
\end{tikzpicture}
\caption{The wave $[u_{0},u_{3}]$ is a composition of the simple waves $[u_{0},u_{1}]$, $[u_{1},u_{2}]$, $[u_{2},u_{3}]$}
\end{figure}

\begin{theorem}[Global structure of solutions]
\label{Th:structure}
Let $u$ be the entropy solution of the Cauchy problem~\eqref{eq:sysnc}-\eqref{eq:inda} under the assumption that each characteristic field is either linearly degenerate or piecewise genuinely nonlinear, and assuming that $\TV(\overline\uu)$ is small enough.
Assume that the source term $g=g(u)$ is Lipschitz continuous.
Then there exists a countable set $\Theta=\{(t_{\ell},x_{\ell})\ :\ \ell\in\nat\}$ 
and a countable family of Lipschitz continuous curves 
\[
\J=\left\{y_{m} : (a_{m} ,b_{m } )\to\R\ ,\ m\in\nat\right\} \ ,
\]
whose graphs cover points of admissible shocks and points of contact discontinuities, such that $u$ is continuous at least outside $\Theta\cup\Graph(\J)$. Moreover, the following holds.
For each curve $\overline{y}=y_{m}\in\J$ and each fixed $t\in(a_{m } ,b_{m })$ with $(t,\overline{y}(t))\notin\Theta$ denote by 
\[
u^{L}:=u^{L}\left(t,\overline{y}(t)\right):=u\left(t,\overline{y}(t)-\right)\ ,
\qquad
u^{R}:=u^{R}\left(t,\overline{y}(t)\right):=u\left(t,\overline{y}(t)+\right)\ .
\]
Then there exist $i$ and $s$ such that $u^{R}=S_{i}[u^{L}](s)$ with $i\in\{1,\dots,N\}$. Moreover
\begin{itemize}
\item If the $i$-th family is linearly degenerate or if $[u^{L},u^{R}]$ is a simple jump of piecewise genuinely nonlinear family $i$ then entropy conditions hold and
\[
u^{L}=\lim_{\substack{(r,x)\to(t,\overline{y}(t))\\x<\overline{y}(s)}} u(r,x)\ ,
\qquad
u^{R}=\lim_{\substack{(r,x)\to(t,\overline{y}(t))\\x>\overline{y}(s)}} u(r,x) \ ,
\qquad
\dot {\overline{y}}(t)=\sigma_{i}(u^{L},u^{R})\ .
\]
In case the $i$-th family is linearly degenerate it is also possible that $u^{L}=u^{R}$ for $s>0$.

\item If $[u^{L},u^{R}]$ is a composition of the waves $[u_{0},u_{1}]$, $[u_{1},u_{2}]$, \dots, $[u_{\ell},u_{\ell+1}]$ then there exist
\[
\overline{y}_{1},\dots,\overline{y}_{p}\in\J, \qquad\text{where $p\leq\ell+1$,}
\]
depending on $(t,\overline{y}(t))$, such that there exists a neighborhood $\U(t)$ of $t$ for which
\[
\overline{y}_{1}(t)=\dots=\overline{y}_{p}(t)\ ,
\qquad
\dot{\overline{y}}_{1}(t)=\dots=\dot{\overline{y}}_{p}(t)=\sigma(u^{L},u^{R})\ ,
\qquad
\overline{y}_{1}(r)\leq \dots\leq \overline{y}_{p}(r)
\]
for $r\in\U(t)$ and
\[
u^{L}=\lim_{\substack{(r,x)\to(t,\overline{y}(t))\\x<\overline{y}_{1}(s)}} u(r,x)\ ,
\qquad
u^{R}=\lim_{\substack{(r,x)\to(t,\overline{y}(t))\\x>\overline{y}_{p}(s)}} u(r,x) \ .
\]
Finally, one can also require that if $\overline{y}_{j}$ and $\overline{y}_{j+1}$ do not coincide in $\U(t)$, then
\[
u_{j}=\lim_{\substack{(r,x)\to(t,\overline{y}(t))\\\overline{y}_{j}(s)<x<\overline{y}_{j+1}(s)}} u(r,x) \ .
\]
\end{itemize}
\end{theorem}
As in~\cite{BreLF,BressanBook,BYu}, the above theorem is proved by approximation by means of a fine convergence result that will be precisely stated later in~\S~\ref{S:qualitative}.
This is why we work under the hypothesis of the convergence
Theorem~\ref{Th:Exloc}.

%


\section{Piecewise constant approximations}
\label{sec:PCA}

In this section we describe the main ingredients in order to construct a piecewise
constant approximation of a solution $u$ to \eqref{eq:sysnc}-\eqref{eq:inda}.

\subsection{The nonconservative Riemann problem}
\label{subsec:RP}

Since we deal with a system that, in general, it is not in conservation form,
we briefly recall the construction of the solution to a Riemann problem in
the homogeneous case, i.e.
\begin{subequations}
\label{eq:RP}
\begin{align}
\label{eq:syshom}
&u_t+A(u)u_x=0&\\
\noalign{\smallskip}
\label{eq:indaRP}
&u(0,x) =\begin{cases}
u^L\quad &\text{if\quad $x<0$\,,}\\
u^R\quad &\text{if\quad $x>0$\,.}
\end{cases}&
\end{align}
\end{subequations}
We refer to~\cite{BB, srp} for the details.

As in the Introduction, we let $A$ be a smooth matrix-valued map,
with eigenvalues given by~\eqref{eq:strhyp}, and right and left
eigenvalues~\eqref{eq:rle}-\eqref{eq:norm}. Since we are interested in solutions
to \eqref{eq:sysnc} with small total variation, it is not
restrictive to assume that there exist
constants $\widehat\lambda_0<\cdots< \widehat\lambda_{N}$ such that
\begin{equation}
\label{eq:hyp}
\widehat\lambda_{k-1}<\lambda_k(u)<\widehat\lambda_k\,,\qquad
\forall~u\,,\quad k=1,\dots,N\,.
\end{equation}
Given any continuous function
$f:I\subset\real\to\real$, and any interval $[a,b]\subset
I$, we will denote the {\it lower convex envelope} and the {\it upper
concave envelope} of~$f$ on~$[a,b]$, respectively, as
\be
\label{eq:conv}
\conv_{[a,b]}f(x)\doteq
\inf\Big\{
\theta f(y)+(1-\theta) f(z): \theta\in[0,1]\,, \ y,z\in[a,b]\,, \
x=\theta y+(1-\theta)z
\Big\}\,,
\ee
and
\be
\label{eq:conc}
\conc_{[a,b]}f(x)\doteq
\sup\Big\{
\theta f(y)+(1-\theta) f(z): \theta\in[0,1]\,, \ y,z\in[a,b]\,, \
x=\theta y+(1-\theta)z
\Big\}\,.
\ee
We will simply write $\conv f,\, \conc f$, whenever
there is no ambiguity on the interval $[a,b]$ taken in consideration.
As usual, in order to contruct a solution to~\eqref{eq:RP}, the
basic step consists in constructing the {\it elementary curve of the
$k$-th family} $(k=1,\dots,N)$ for every given left state $u^L$, which
is a one parameter curve of right states $s\mapsto T_k[u^L](s)$
with the property that the Riemann problem having initial data $(u^L,
u^R)$, $u^R\doteq T_k[u^L](s)$, admits a vanishing viscosity solution
consisting only of waves of the $k$-th characteristic family.
In order to construct such a curve, we look for travelling waves
solutions to the parabolic system
\be
\label{eq:parab}
u_t+A(u)u_x = u_{xx}\,,
\ee
solutions to \eqref{eq:parab}
of the form $u(t,x)=\phi(x-\sigma\,t)$, for some constant $\sigma$.
The profile $\phi$ satisfies the second order ODE
\be
\big(A(\phi)-\sigma\big)\,\phi'=\phi''\,,
\nonumber
\ee
which can be written as a first order 
system of ODEs on the space $\real^N\times\real^N\times\real$:
\be
\label{eq:vtw}
\begin{cases}
\dot u=v\,,
\\
\dot v=\big(A(\phi)-\sigma\big)\,v\,,
\\
\dot\sigma=0\,.
\end{cases}
\ee
Applying the Center Manifold Theorem, we get that in a neighborhood
of a given equilibrium point $(u_0,\,0,\,\lambda_k(u_0))\in
\real^N\times\real\times\real$ for~\eqref{eq:vtw}
there exists an $N+2$-dimensional center manifold $\MC_k$
which is locally invariant under the flow
of~(\ref{eq:vtw}). Introducing the coordinates
$$
v_h\doteq\scalar{l_h(u_0)}{v}\,,\qquad\quad
h=1,\dots,N\,,
$$
of a vector $v\in\real^N$ relative to the basis $r_1(u_0),\dots$
$\dots,r_N(u_0)$, one can parameterize $\MC_k$ in terms of the variables
$u,\,v_k,\,\sigma$, namely
\be
\MC_k=\big\{
(u,v,\sigma)~;~~ v=v_k\, \widetilde r_k (u,v_k,\sigma)
\big\}
\ee
for suitable smooth vector functions $(u,v_k,\sigma) \mapsto\widetilde
r_k (u,v_k,\sigma)$ defined on a neighborhood of $(u_0,0,\lambda_k(u_0))$,
that satisfy
\be
\label{eq2:ri}
\widetilde r_k \big(u_0,0,\sigma\big) = r_k(u_0) \qquad
\forall~\sigma\,,
\ee
and are normalized so that
\be
\label{normrt}
\scalar{l_k(u_0)}{\widetilde r_k (u,v_k,\sigma)} = 1 \qquad
\forall~u\,,\,v_k\,,\,\sigma\,.
\ee
By construction, $\MC_k$ contains all bounded viscous traveling 
profiles with speed close to $\lambda_k(u_0)$.
Thus, we
can rewrite the linearized equations for (\ref{eq:vtw})  at $(u_0,0,\lambda_k(u_0))$
on the manifold $\MC_k$, and obtain
a system on the space $\real^N\times\real\times\real$:
\be
\label{rtp}
\begin{cases}
  u_x = v_k \widetilde r_k (u,v_k,\sigma)\,,\\
  v_{k,x} = v_k \big(\, \widetilde \lambda_k (u,v_k,\sigma) - \sigma\,
  \big)\,,\\
  \sigma_{x} = 0\,,
\end{cases}
\ee
where
\be
  \label{eq2:tlam}
  \widetilde \lambda_k (u,v_k,\sigma) \doteq \scalar{l_k(u)}{A(u)\,
    \widetilde r_k (u,v_k,\sigma)}\,.
\ee
Because of the normalization (\ref{normrt}), the smooth scalar
function $(u,v_k,\sigma) \mapsto\widetilde \lambda_k (u,v_k,\sigma)$
satisfies the identity
\be
\label{eq2:lk}
  \widetilde \lambda_k \big(u_0,v_k,\sigma \big) =
  \lambda_k(u_0)\qquad 
  \forall~v_k\,,\,\sigma\,.
  \ee
Next, given a left state $u^L$ in a neighborhood  of $u_0$ 
and $0<s<<1$, in connection with the
equations~(\ref{rtp}) describing the evolution
of traveling profiles on the manifold $\MC_k$
we associate the integral system 
\begin{equation}
  \label{eq2:Ti}
  \begin{cases}
    u(\tau) = u^L + \displaystyle{\int_0^\tau} \widetilde r_k
    \big( u(\xi), v_k(\xi), \sigma(\xi)
    \big)~d\xi\,,\\
    \noalign{\medskip}
    v_k(\tau) = \widetilde F_k\big(\tau;\,u,v_k,\sigma\big) -
    \conv_{[0,s]} \widetilde F_k\big(\tau;\,u,v_k,\sigma\big)\,,\\
    \noalign{\medskip}
    \sigma(\tau) = \dtau \conv_{[0,s]} \widetilde
    F_k\big(\tau;\,u,v_k,\sigma\big)\,,
  \end{cases}
  \qquad
  0\leq\tau \leq s\,,
\end{equation}
where $\tau\mapsto\widetilde F_k(\tau;\,u,v_k,\sigma)$ is the
{\it``reduced flux function''} associated to (1.13) defined, by
\begin{equation}
  \label{eq2:tFi}
  \widetilde F_k(\tau;\,u,v_k,\sigma) \doteq\int_0^\tau \widetilde \lambda_k
  \big( u(\xi), v_k(\xi), \sigma(\xi)
  \big)~d\xi\,.
\end{equation}
In~\cite{BB} it is shown that, for $s$ sufficiently small, the transformation
defined by the right-hand side of~(\ref{eq2:Ti}) maps a domain of
continuous curves $\tau\mapsto (u(\tau),v_k(\tau),\sigma(\tau))$ into
itself, and is a contraction w.r.t. a suitable weighted norm.  Hence,
for every $u^L$ in a neighborhood $\U_0$ of $u_0$, the transformation
defined by~(\ref{eq2:Ti}) admits a unique fixed point
\bel{E:Rfp}
\tau\mapsto
\big(\overline u(\tau;\,u^L,s),\ \overline v_k(\tau;\,u^L,s),\ \overline \sigma(\tau;\,u^L,s)\big)
\qquad\tau\in[0,s]\,,
\ee
which provides a Lipschitz continuous solution to the integral system (\ref{eq2:Ti}).
The elementary curve of right states of the $k$-th family 
issuing from $u^L$ is then defined as the terminal value at $\tau=s$
of the $u$-component of the solution to the integral system (\ref{eq2:Ti}),
i.e. by setting 
\be
\label{eq:Tdef}
T_k[u^L](s)\doteq u(s;\,u^L,s)\,.
\ee
For the sake of convenience, we denote
\be
\label{eq:sigmaFdef}
\begin{aligned}
\sigma_k[u^L](s,\tau) &\doteq \sigma(\tau;\,u^L,s)\,,\\
\noalign{\smallskip}
\widetilde F_k[u^L](s,\tau) &\doteq \widetilde F_k \big( \tau;
u(\cdot\ ;u^L,s), v_k(\cdot; u^L,s), \sigma(\cdot ; u^L,s) \big)
\end{aligned}
\ee
For negative values $s<0$, $|s|<<1,$ one replaces in (\ref{eq2:Ti})
the lower convex envelope of~$\widetilde F_k$ on the interval $[0,s]$
with its upper concave envelope on $[s,0]$, and then constructs the
curve $T_k[u^L]$ and the map $\sigma_k[u^L]$ exactly in the same way
as above looking at the solution of the integral system (\ref{eq2:Ti})
on the interval $[s,0]$.  In such a way, given any
pair of states~$u^L, u^R$ with $\vert u^L-u_0\vert, \vert u^R-u_0\vert
<<1$, if $u^R=T_k[u^L](s)$, for some wave size~$s$, then the
self-similar solution to the Riemann problem with initial
data~$(u^L,u^R)$, determined by the vanishing viscosity approximation
\eqref{eq:vv} as $\eps\to 0+$, is given by the piecewise continuous
function
\begin{equation}
  \label{eq:solRPi}
  u (t,x) =
  \begin{cases}
    u^L\quad &\text{if\quad $x/t <\sigma_k[u^L](s,\,0)\,,$}\\
    \noalign{\smallskip}
    T_k[u^L] (\tau)\quad &\text{if\quad $x/t = \sigma_k[u^L](s,\,\tau)$
    \quad for some $\tau\in\mathcal{I}$\,,}\\
    \noalign{\smallskip}
    u^R\quad &\text{if\quad $x/t >\sigma_k[u^L] (s,\,s)\,,$}
  \end{cases}
\end{equation}
\begin{remark}
\label{rem2:consform}
{\rm
If the system (\ref{eq:sysnc}) is in conservation form, i.e. in the case
where $A(u)=DF(u)$ for some smooth flux function $F$, the general
solution of the Riemann problem provided by (\ref{eq:solRPi}) is a
composed wave of the $k$-th family containing a countable number of
rarefaction waves and contact-discontinuities or compressive shocks
which satisfy the Liu admissibility condition~\cite{tplrp2p2, tplrpnpn}.
Namely, the regions where the
$v_k$-component of the solution to~(\ref{eq2:Ti}) vanishes correspond
to rarefaction waves if the $\sigma$-component is strictly increasing
and to contact discontinuities if the $\sigma$-component is constant,
while the regions where the $v_k$-component of the solution
to~(\ref{eq2:Ti}) is different from zero correspond to compressive
shocks.
}
\end{remark}
In view of the considerations of Remark~\ref{rem2:consform},
we will extend the standard terminology adopted for the
elementary waves that are present in the solution
of an hyperbolic system of conservation laws 
to  the general case of non conservative systems.
Thus, 
we will say that any (vanishing viscosity) solution of the Riemann
problem for (\ref{eq:sysnc}) of the form (\ref{eq:solRPi}) is a {\it
centered rarefaction wave} of the $k$-th family whenever $u^R\in
R_k[u^L](s)$ for some wave size~$s$ such that $\tau\mapsto
\sigma_k[u^L](s,\tau)$ be strictly increasing on $[0,s]$, $s>0$ (or
strictly decreasing on $[s,0]$ if $s<0$), while we will say that any
(vanishing viscosity) solution of a Riemann problem for (\ref{eq:sysnc}) of
the form
$$
u(t,x)=\begin{cases}
u^L\quad &\text{if\quad $x<\lambda t$\,,}\\
u^R\quad &\text{if\quad $x>\lambda t$\,,}
\end{cases}
$$
is an {\it admissible shock wave} of the $k$-th
family when $u^R=T_k[u^L](s)$ and
$\sigma_k[u^L](s,0)=\sigma_k[u^L](s,s)=\lambda$.
Once we have constructed the elementary curves $T_k$ for each $k$-th
characteristic fa\-mi\-ly, the {\it vanishing viscosity solution} of a
general Riemann problem for (\ref{eq:sysnc}) is then obtained by a standard
procedure observing that the composite mapping
\be
\Phi(s_1,\dots,s_N) [u^L]\doteq T_N\Big[ T_{N-1}\Big[\cdots
    \big[T_1[u^L](s_1)\big]\cdots\Big](s_{N-1})
\Big](s_N)\doteq
u^R\,,
\label{eq2:RP1}
\ee
is one-to-one from a neighborhood of the origin onto a neighborhood
of $u^L$. This is a consequence of the fact that the curves $T_k[u]$
are tangent to $r_k(u)$ at zero $s=0$~\cite{BB, srp}.
Therefore, 
we can uniquely determine intermediate states $u^L\doteq\omega_0,$
$\omega_1,$ $\dots,$ $\omega_N\doteq u^R$, and wave sizes $s_1, \dots,
s_N,$ such that there holds
\begin{equation}
\label{eq2:RP2}
\omega_k =T_k[\omega_{k-1}](s_k)\quad\qquad k=1,\dots,N\,,
\end{equation}
provided that the left and right states $u^L, u^R$ are sufficiently
close to each other. Each Riemann problem 
with initial data 
\be
\label{elemriem}
\overline u_k(x) = 
\begin{cases}
\omega_{k-1} &\text{if \ \ $x<0$,}\\
\omega_k &\text{if \ \ $x>0$,}
\end{cases}
\ee
admits a vanishing viscosity solution of {\it total size} $s_k$,
containing a sequence of rarefactions and Liu admissible
discontinuities of the $k$-th family.  Then, because of the uniform
strict hyperbolicity assumption~(\ref{eq:hyp}), the general solution of
the Riemann Problem with initial data $\big(u^L,\,u^R\big)$ is
obtained by piecing together the vanishing viscosity solutions of the
elementary Riemann problems (\ref{eq:sysnc}) (\ref{elemriem}).  Throughout
the paper, with a slight abuse of notation, we shall often call $s$ a
wave of (total) size~$s$, and, if $u^R=T_k[u^L](s)$, we will say that
$(u^L,\,u^R)$ is a wave of size $s$ of the $k$-th characteristic
family.

\subsection{The algorithm}
\label{Ss:frontTr}

Now we briefly describe the algortihm we use in order to construct
a piecewise constant approximate solution to ~\eqref{eq:sysnc}-\eqref{eq:inda}.
First of all let us recall what a front tracking solution to an homogeneous
hyperbolic system is (see~\cite{AMfr} for details).
\begin{definition}
\label{def:ft}
Let $\eps>0$ and an interval $I \subset \R$ be fixed, and let
$A=A(u)$, $u\in\R^N$, be a smooth hyperbolic $N\times N$ matrix.
We say that a continuous map $u:
I\mapsto \elleuno_{loc}(\real;\,\real^N)$, is an
$\eps$-approximate front tracking solution to~\eqref{eq:syshom}
if the following conditions hold:
\begin{enumerate}
\item
  As a function of two variables, $u=u(t,x)$ is piecewise constant
  with discontinuities occurring along finitely many straight lines in
  the $t$-$x$ plane. Jumps can be of two types: elementary wave-fronts
  and non-physical wave-fronts, denoted, respectively, as $\E$ and
  $\NP$.  Only finitely many wave-fronts interactions occur, each
  involving exactly two incoming fronts.
\item
  Along each elementary front $x=x_\alpha(t)$, $\alpha\in \E$, the
  values $u^L\doteq u(t,\,x_\alpha-)$ and $u^R\doteq u(t,\,x_\alpha+)$
  satisfy the following properties.  There exists some wave size
  $s_\alpha$ and some index $k_\alpha\in\{1,\dots,N\}$ such that
  \begin{equation}
    \label{eq3:appphy}
    u^R = T_{k_\alpha} [u^L](s_\alpha)\,.
  \end{equation}
  Moreover, the speed $\dot x_\alpha$ of the wave-front satisfies
  \begin{equation}
    \label{shockspeederr}
    \Big| \dot x_\alpha - \sigma_{k_\alpha}[u^L](s_\alpha,\tau)\Big|
    \leq \rarbound\,,
    \qquad\forall~\tau\in [0,\,s_\alpha]\,.
  \end{equation}
\item
  All non-physical fronts $x=x_\alpha(t)$, $\alpha\in {\mathcal N}$
  have the same speed
  \begin{equation}
    \label{npspeed1}
    \dot x_\alpha\equiv\widehat\lambda\,,
  \end{equation}
  where $\widehat\lambda$ is a fixed constant strictly greater than
  all characteristic speeds, i.e. 
  \begin{equation}
  \label{nps}
  \widehat\lambda>\lambda_k(u)\qquad
  \quad\forall~u\in\Omega,\quad k=1,\dots,N\,.
  \end{equation}
  Moreover, the total strength of all non-physical
  fronts in $u(t,\cdot)$ remains uniformly small, namely one has
  \begin{equation}
    \label{npbound}
    \sum_{\alpha\in\NP}
    \big|u(t,x_\alpha+)-u(t,x_\alpha-)\big|\leq 
    \eps
    \qquad\quad\forall~t\geq 0\,.
  \end{equation}
\end{enumerate}
\end{definition}
In order to construct piecewise constant approximations
to~\eqref{eq:sysnc}-\eqref{eq:inda}, we follow the approach of~\cite{CP}, and
construct a local solution to~\eqref{eq:sysnc}-\eqref{eq:inda} by means
of a fractional step algorithm combined with a front tracking method.
In order to do this we assume that assumption \textbf{(G)} at~\pageref{Ass:G}
holds.
Hence, once two sequences
\[
\{ \tau_{\nu} \}_{\nu\in\nat},\ 
\{ \eps_{\nu} \}_{\nu\in\nat}  \ ,
\qquad
0<\tau_{\nu}\leq \eps_\nu \downarrow 0 \ ,
\]
are given, we fix $\nu\in\nat$ and we proceed in this way
in order to construct and $\eps_\nu$-approximate fractional-step
approximation $u_\nu=u_{\eps_{\nu}}$ of the solution.
Fist of all, we approximate the initial datum $\overline{u}$ by means of a piecewise
constant function $\overline{u}_{\nu}$ such that
$$
\TV \overline{u}_{\nu} \leq \TV \overline{u}\,, \qquad
\Vert \overline{u}_{\nu}-\overline{u}\Vert_{L^1} \to 0
\quad\text{as $\nu\uparrow\infty$}\,.
$$
Then, we take a suitable approximation
$g_{\nu}$ of $g$ piecewise contant w.r.t. $x$, i.e., following~\cite[\S~3]{CP},
we let
\be
\label{eq:gnu1}
g_\nu(t,x,v)\doteq \sum_{j\in\Z} \chi_{[j\eps_\nu,(j+1)\eps_\nu[}(x)
g_j(t,v)\,,
\ee
where $\chi_I$ is the characteristic function of the set $I$, and
\be
\label{eq:gnu2}
g_j(t,v) = \frac{1}{\eps_\nu} \int_{j\eps_\nu}^{(j+1)\eps_\nu}
g(t,x,v)\, dx\,.
\ee
Then, the algorithm that
leads to the construction of the approximation $u_{\nu}=u_{\nu}(t,x)$
essentially consists of the following steps. 
\begin{enumerate}
\item
We apply a front tracking algorithm as described in~\cite{AMfr}, which we refer
to, to construct
an $\eps_\nu$-approximate front tracking solution in the sense of
Definition~\ref{def:ft} in the time interval $]0,\tau_{\nu}[$.
\item
At $t=\tau_{\nu}$ we correct the term $u_{\nu} (\tau_{\nu}-,\cdot)$
by setting
$$
u_{\nu} (\tau_{\nu}+,\cdot)= u_{\nu} (\tau_{\nu}-,\cdot)+
\tau_{\nu} g_{\nu} \big(  \tau_{\nu} , \cdot, u_{\nu} (\tau_{\nu}-,\cdot) \big))\,,
$$
which turns out to be piecewise constant by construction.
\item
In general, once $u_{\nu} (n\tau_{\nu}+,\cdot)$, $n\geq 1$, is given,
we again use the algorithm in~\cite{AMfr} to construct
an $\eps_\nu$-approximate front tracking solution in the time interval
$]n\tau_{\nu}, (n+1)\tau_{\nu}[$.
\item
Similarly to what done above at $t=\tau_{\nu}$,
at $t=(n+1)\tau_{\nu}$ we correct the term $u_{\nu} ((n+1)\tau_{\nu}-,\cdot)$
by setting
$$
u_{\nu} \big( (n+1)\tau_{\nu}+,\cdot \big)= u_{\nu} \big(
(n+1) \tau_{\nu}-,\cdot \big)+
\tau_{\nu} g_{\nu} \big(  (n+1)\tau_{\nu} , \cdot, u_{\nu}
\big( (n+1)\tau_{\nu}-,\cdot \big) \big))\,.
$$
\end{enumerate}
 We stress that, in the construction described above, nonphysical waves are implicitly restarted at each time step: the corresponding jumps are solved using physical waves. 
As it is usual with such algorithms, the main difficulties we have to face are to
\begin{itemize}
\item
bound uniformly the total variation of $u_{\nu} (t,\cdot)$ in order to get
compactness of the approximating sequence;
\item
let the number of the fronts to remain bounded in any time interval $[0,t]$.
\end{itemize}
We will briefly discuss how to overcome the first difficulty in
Subsection~\ref{Ss:PreviousEstimates}, taking advantage of the results contained
in~\cite{AMfr,sie,CP}. Regarding the second difficulty,
using the arguments contained in~\cite[Subsection~6.2]{AMfr},
it can be easily seen that
the number of wave fronts stays bounded in each time interval $[k\tau_\nu,
(k+1)\tau_\nu[$, and their number depends on the parameter $\eps_\nu$
and on the total variation of $u_\nu(t,\cdot)$ which remains uniformly bounded.

\subsection{Evolution / interaction estimates}
\label{Ss:PreviousEstimates}
\indent

In correspondence of a sequence $\{\eps_\nu\}_{\nu\geq 1}\subset\R^{>0}$,
$\eps_\nu\to 0$,
and following~\cite{Glimm}, in this subsection we will define the interaction potential
and give the interaction estimates that will allow us to perform uniform
bounds on the total variation of an $\eps_\nu$ frotn tracking approximate
solution. To this purpose,
following~\cite[Definition~3.5]{sie}, we first introduce a definition
of quantity of interaction between wave-fronts of an approximate
solution.
\begin{definition}
\label{def:qi}
{\rm
Consider 
two interacting wave-fronts of sizes $s', s''$ ($s'$ located on 
the left of $s''$), belonging to the $k', k''\in\{1,\dots, N+1\}$-th
characteristic family, respectively, and let 
 $u^L,\,u^M,\,u^R$,
 denote the left, middle and right
states before the interaction.
We say that the  {\it
amount of interaction} ${\I}(s^\prime,\,s^{\second})$
between $s^\prime$ and $s^{\second}$ is the quantity defined as
follows.
\begin{enumerate}
  \item
    If $s^\prime$ and $s^{\second}$ belong to different characteristic
    families, i.e. if $k''<k'\leq N+1$, then set
    \begin{equation}
      \label{eq2:intdf}
	{\I}(s^\prime,\,s^{\second}) \doteq \vert s^\prime
	s^{\second} \vert\,.
    \end{equation}
  \item
    If $s^\prime$ and $s^{\second}$ belong to the same $k\, (\leq\!\!
    N)$-th characteristic family $(k\doteq k'=k'')$, i.e. if
    $u^M=T_k[u^L](s'),\, u^R=T_k[u^M](s'')$, let $\widetilde
    F^{\prime,L} \doteq \widetilde F_k[u^L](s',\,\cdot\,)$ and
    $\widetilde F^{\second,M}\doteq \widetilde
    F_k[u^M](s'',\,\cdot\,)$ be the reduced flux with starting point
    $u^L$, $u^M$, evaluated along the solution of~(\ref{eq2:Ti}) on
    the interval $[0,s^\prime]$, and $[0,s^\second]$, respectively
    (cfr. def.~(\ref{eq:sigmaFdef})).  Then, assuming that~$s\geq 0$, we
    shall distinguish three cases.
    \begin{enumerate}
      \item if $s^\second\geq 0$\, set:
        \be
	\begin{aligned}
	  \label{eq2:qi1}
	\I(s^\prime,\,s^{\second}) &\doteq \int_0^{s^\prime}
	  \left\vert \conv_{[0,\,s^\prime]} \widetilde F^{\prime,L}(\xi) -
	  \conv_{[0,\,s^\prime+s^\second]}
	\widetilde F^{\prime,L}
	  \!\cup\! \widetilde F^{\second,M}
	(\xi) \right\vert d\xi\\
	  \noalign{\smallskip}
	  &\quad + \int_{s^\prime}^{s^\prime+s^\second} \left\vert
	  \widetilde F^{\prime,L} (s^\prime)+\conv_{[0,\,s^\second]} 
	  \widetilde F^{\second,M} (\xi-s^\prime) \right.\\
	  \noalign{\smallskip}
	  &\quad\quad \left. -
	  \conv_{[0,\,s^\prime+s^\second]} \widetilde F^{\prime,L}
	  \!\cup\! \widetilde F^{\second,M} (\xi) \right\vert d\xi\,,
	\end{aligned}
	\ee
	where $\widetilde F^{\prime,L} \!\cup\! \widetilde
        F^{\second,M}$ is the function defined on $[0,\,s'+s'']$ as
	\begin{equation}
	  \label{eq2:FpFs}
	  \widetilde F^{\prime,L} \!\cup\! \widetilde F^{\second,M} (s) \doteq
	  \begin{cases}
	    \widetilde F^{\prime,L}(s) \quad &\text{\testo{if}\quad $s\in
	    [0,s^\prime]$\,,}\\
	    \noalign{\smallskip}
	    \widetilde F^{\prime,L}(s^\prime) + \widetilde
	    F^{\second,M}(s-s^\prime) \quad &\text{\testo{if}\quad $s\in
	    [s^\prime, s^\prime+s^\second]$\,.}
	  \end{cases}
	\end{equation}
        \item \ if $-s^\prime\leq s^\second\!<0$\, set:
	\be
	\begin{aligned}
	  \label{eq2:qi2}
	  {\I}(s^\prime,\,s^{\second}) &\doteq
	  \int_0^{s^\prime+s^\second} \left\vert \conv_{[0,\,s^\prime]}
	  \widetilde F^{\prime,L}(\xi) - \conv_{[0,\,s^\prime+s^\second]}
	  \widetilde F^{\prime,L}(\xi) \right\vert d\xi\\
	  \noalign{\smallskip}
	  &+ \int_{s^\prime+s^\second}^{s^\prime} \left\vert
	  \conv_{[0,\,s^\prime]} \widetilde F^{\prime,L}(\xi) -
	  \conc_{[s^\prime+s^\second,\,s^\prime]} F^{\prime,L}(\xi)
	  \right\vert d\xi\,.
	\end{aligned}
        \ee	
      \item \ if $s^\second<-s^\prime$\, set: 
	\be
	\begin{aligned}
	  \label{eq2:qi3}
	  {\I}(s^\prime,\,s^{\second}) &\doteq
	  \int_{s^\prime+s^\second}^0 \left\vert \conc_{[s^\second,\,0]}
	  \widetilde F^{\second,M}(\xi-s') - \conc_{[s^\second,\,-s^\prime]}
	  \widetilde F^{\second,M}(\xi-s') \right\vert d\xi
	 \\
	  \noalign{\smallskip}
	  &+ \int_0^{s^\prime} \left\vert \conc_{[s^\second,\,0]}
	  \widetilde F^{\second,M}(\xi-s') -
	  \conv_{[-s^\prime,\,0]}\widetilde F^{\second,M}(\xi-s')
	  \right\vert d\xi\,.
	\end{aligned}
        \ee
       \end{enumerate}
	In the case where $s^\prime<0$, one replaces in
	(\ref{eq2:qi1})-(\ref{eq2:qi3}) the convex envelope with the
	concave one, and vice-versa.
\end{enumerate}
}
\end{definition}
\begin{remark}
\label{rem4:amshocks}
{\testo
By Remark~\ref{rem2:consform} one can easily verify that, in the
conservative case, if $s', s''$ are both shocks of the $k$-th family
that have the same sign, then the amount of interaction in
(\ref{eq2:qi1}) takes the form
$$
\I(s', s'')=\big|s' s''\big| \Big|\sigma_k[u^L, u^M]-\sigma_k[u^M, u^R]\Big|\,,
$$
i.e. it is precisely the product of the strength of the waves times
the difference of their Rankine Hugoniot speeds.
}
\end{remark}
Now, whenever a $\eps$-approximate front tracking solution
$u^\eps=u^\eps(t,x)$ to~\eqref{eq:syshom} is given,
we define the interaction potential (see~\cite[(4.2)]{sie})
\be
\label{eq:ip}
\Q (u^\eps(t,\cdot)) = \sum_{\substack{i<j\\ x'>x''}}
\big\vert s'_{x',i}s''_{x'',j} \big\vert +
\frac{1}{4} \sum_{x',x'',i} \int_0^{\vert s_{x',i}\vert} \int_0^{s_{x'',i}}
\big\vert \sigma_{x',i}(\tau')-\sigma_{x'',i}(\tau'') \big\vert\, d\tau'd\tau''\,,
\ee
where $s_{x,k}$ is the size of the wave of the $k$-th characteristic family
at $x$, and $\sigma_{x,k}(\tau)$ is its speed as it is defined at~\eqref{eq2:Ti}.
Moreover we let
\be
\label{eq:V}
\mathcal{V} (u^\eps(t,\cdot)) = \sum_{x,i} \vert s_{x,i}\vert
\ee
With these definitions, the following result holds (see~\cite[Proposition~4.1]{sie}):
\begin{proposition}
\label{pro:ie}
There exists $\dH>0$ such that, if
$u^\eps=u^\eps(t,x)$ is an $\eps$-approximate front tracking solution
to~\eqref{eq:syshom}-\eqref{eq:inda} with $\TV \overline{u}<\dH$,
the the following holds. There exists constants $c,C_1>0$ such that,
whenever two wave fronts $s',s''$ interact, then
$$
\Delta \mathcal{Q} \leq -c\I (s',s'')\,,
$$
and, moreover, the functional
\be
\label{eq:Ups}
t\mapsto \Upsilon (u^\eps(t,\cdot)) = \mathcal{V} (u^\eps(t,\cdot))
+ C_1\Q (u^\eps(t,\cdot))
\ee
is decreasing.
\end{proposition}

\section{Existence and convergence of approximations}
\label{S:convergenceAndExistence}

In this section we prove that the approximations constructed in \S~\ref{sec:PCA} converge to the entropy solution of the Cauchy problem~\eqref{eq:sysnc}-\eqref{eq:inda}.
We first prove rough estimates that ensure the local-in-time convergence, as stated in Theorems~\ref{T:localConv}-\ref{T:localEst} below, which yield Theorem~\ref{Th:Exloc}.
Uniqueness is proved roughly following the lines of~\cite{AG2}.

\subsection{Local in time existence of time-step approximations}

Let $\Upsilon$ be the functional introduced in~\eqref{eq:Ups}.

\begin{theorem}
\label{T:localConv}
There exist $\dssb, T>0$ such that for initial data $\overline{\uu}$ in the closed domain
\[
\mathfrak D_{p}(\dssb):=\left\{\uu\in L ^{1}(\R;\R^{N})\cap \BV(\R;\R^{N})\text{ piecewise constant s.t.~} \Upsilon(\uu)\leq\dssb\right\}
\]
the algorithm described in \S~\ref{sec:PCA} defines for $t\in[0,T]$ and for every $\nu$ an approximating function
\bel{E:domt}
\ww_{\nu}(t,\cdot)\in\mathfrak D_{p}\left( \dssb+Gt\right) \quad \text{where $C,G$ only depend on $A$ and $g$.}
\ee
This approximating function $u_{\nu}$ satisfies the following comparison estimate with the viscous semigroup~\cite{Ch1} $\vSC{t}{h}[\cdot]$ of the Cauchy problem~\eqref{eq:sysnc}-\eqref{eq:inda} starting at time $h$: there is a function $o(s)$ depending only on $A$, $g$, $\dssb$, $T$ such that $o(s)/s\to0$ if $s\to0$ and such that for $n\in\nat$
\bel{E:uniqEst}
\big\lVert \ww_{\nu}(n \tau_{\nu}+,\cdot) - \vSC{n \tau_{\nu}}{(n-1)\tau_{\nu}} \left[\ww_{\nu}((n-1)\tau_{\nu}+,\cdot)\right]\big\rVert_{L^{1}} \leq \OL(1)\left(o(\tau_{\nu}) +\eps_{\nu}\tau_{\nu}
\right)
\ .
\ee
\end{theorem} 

\paragraph{Introduction to the proof.}
Before the proof, we briefly remind our notation and previous results that we need. We denote by $\ftS{t}{h}$ the wave-front tracking approximation of the semigroup $\vSB{t}{h}[\cdot]$ relative to the homogeneous system, constructed by vanishing viscosity~\cite{BB}, where the `initial datum' is fixed at time $h\leq t$ rather than at $h=0$.

We exploit the definition in \S~\ref{sec:PCA} of the approximation
\bel{E:approx}
\begin{split}
\ww_{\nu}(t,\cdot) &=\ftS{ t}{(n-1)\tau_{\nu}}\left[\ww_{\nu}((n-1) \tau_{\nu}+,\cdot) \right]
\qquad\qquad\text{for $(n-1)\tau_{\nu}<t<n\tau_{\nu}$, $n\in\N$}
\\
\ww_{\nu}( n\tau_{\nu}+,\cdot) 
&\equiv 
 \ftS{ n\tau_{\nu}}{(n-1)\tau_{\nu}}\left[\ww_{\nu}((n-1) \tau_{\nu}+,\cdot) \right]+ \tau_{\nu}g\left(  n\tau_{\nu},\cdot, \ftS{ n\tau_{\nu}}{(n-1)\tau_{\nu}}\left[\ww_{\nu}( n-1) \tau_{\nu}+,\cdot) \right] \right)
\\
&\equiv \ww_{\nu}( n\tau_{\nu}-,\cdot) + \tau_{\nu}g\left( n \tau_{\nu},\cdot, \ww_{\nu}( n\tau_{\nu}-,\cdot) \right)
\end{split}
\ee
relative to the balance law with the initial condition $\ww_{\nu}(0+,\cdot)\equiv\bar u(\cdot)$.
We recall that 
\[
\text{if $\overline \ww\in \mathfrak D_{p}(\dssb)$, for $\dssb\leq \dH$ small enough as in~\cite{AMfr}, }
\]
then
\ba
\label{E:rgabab}
&\Upsilon\left(\ftS{t}{h}\overline\ww\right)\leq\Upsilon(\overline\ww)
\qquad \forall0\leq h\leq t\leq \tau_{\nu}<T&
&\text{see~\cite[(6.4)]{AMfr} or Proposition \ref{pro:ie} above}&
\\
\label{E:ftest1}
& \lVert \ftS{n \tau_{\nu}}{(n-1)
\tau_{\nu}} \overline{\ww} -\vSB{n \tau_{\nu}}{(n-1) \tau_{\nu}} \overline{\ww}
\rVert_{L^{1}} \lesssim (1+\dssb)\varepsilon_{\nu}\tau_{\nu}&
&\text{see \cite[(3.5)]{AMfr}}&
\\
\label{E:ftest2}
& \lVert \ftS{t+s}{t}\overline{\ww} -\overline{\ww}\lVert \leq L s &
&\text{see \cite[(1.23)]{AMfr}}&
\ea
We also borrow the following lemma from~\cite[Lemmas~2.1-2]{AG2}, given in a similar setting. Of course we could state it similarly also localizing in space the estimates.
We remind that $\Upsilon, \Q$ are the functionals introduced in~\eqref{eq:ip}-\eqref{eq:Ups} while $\ell_{g}$ and $\alpha$ are as in the assumption (G) on the source term at Page~\pageref{Ass:G}.

\begin{lemma}
\label{L:timeEst}
Let $t>0$.
If $0<\dssb<\dH$ and $\overline{\ww},\overline{\uu}$ are piecewise constant with $\Upsilon(\overline{\uu})+\Upsilon(\overline{\ww})\leq\dssb$ then \[\vv(x):=\overline{\uu}(x)+\tau g_{\nu}(t,x,\overline{\ww}(x))\] satisfies  for $G= \max\{  \ell_{g}\TV(\overline\ww)+\TV(\overline\uu)+\norm{\alpha}_{L^{1}}; 1\} $ the inequalities
\bel{E:rwrg}
\left| \TV(\vv)- \TV(\overline\uu)\right|\lesssim G \tau
\qquad
\left|\Q(\vv) - \Q(\overline\uu)\right|
\lesssim G^{2} \tau
\qquad
\left|\Upsilon(\vv) - \Upsilon(\overline\uu)\right|
\lesssim
 G^{2}\tau \ .
\ee
\end{lemma}
\begin{proof}
We remind the idea of the proof from~\cite[Lemma~2.1]{AG2}  for completeness. 
Suppose either $\overline{\uu}(x)$ or $ \vv(x)$ has a jump at $x$.
Denoting by $\Phi(\cdot)[\cdot]$ the map defined at~\eqref{eq2:RP1} for the Riemann problem, set $\widehat \sigma$ by the relation
\[
\overline\ww(x+) = \Phi(\widehat\sigma)[ \overline\ww(x-)] \ .
\]
Define then $\sigma'$ so that the following diagram commutes:
\[\begin{CD}
\overline{\uu}(x-)   @>\text{source}>> \vv(x-):=\overline{\uu}(x-)+\tau g_{\nu}(t,x-,\overline{\ww}(x-))  \\
@VV\sigma V        @V\Rightarrow V\sigma' V\\
\overline{\uu}(x+)=\Phi(\sigma)[\overline{\uu}(x-)]   
@>\text{source}>>  
\begin{split} \vv(x+)&:=\overline{\uu}(x+)+\tau  g_{\nu}(t,x+,\overline{\ww}(x+))\\
&\equiv\Phi(\sigma')[ {\vv}(x-)]    \end{split}
\end{CD}
\]
\paragraph{Estimate on the total variation}
The first estimate immediately follows since
\bas
|\TV(\vv)-\TV(\overline\uu)|&\leq \TV(\vv-\overline\uu)=\tau\cdot \TV( g_{\nu}(t,x+,\overline{\ww}(x+)))
\\
&\leq (\ell_{g} \TV(\overline\ww) + \norm{\alpha}_{L^{1}}   )\tau \ .
\eas

\paragraph{Estimate on $|\sigma'-\sigma|$}
Set 
\ba\label{E:roijgr}
d=g_{\nu}(t,x+,\overline{\ww}(x+))-g_{\nu}(t,x-,\overline{\ww}(x-))
\quad\Rightarrow\quad
|d|\leq \int_{(j-1)\eps_\nu}^{(j+1)\eps_\nu}\!\!\!\!\!|\alpha|+\ell_{g}\,\,\mathrm{Lip}\Phi\,\, \widehat\sigma \ .
\ea
Notice that the difference $|\sigma'-\sigma|$ is a function of $\sigma,\tau,d$ which identically vanishes both when $\sigma=d=0$ and when $\tau=0$, as the two rows / columns of the commutative diagram above collapse.
One can thus estimate $\Psi(\sigma,\tau,d)=|\sigma'-\sigma|$ by calculus similarly to~\cite[Lemma~2.5]{BressanBook}, since $\Psi_{\sigma}(\sigma,0,d)=\Psi_{d}(\sigma,0,d)=0$ and  $\Psi_{\tau}(0,\tau,0)=0$:
\ba
\notag
|\sigma'-\sigma|&=\left|\int_{0}^{1}\left(\sigma\Psi_{\sigma}+d\Psi_{d}\right)(z\sigma,\tau,zd)\,dz\right|
=\left|\int_{0}^{1}\int_{0}^{\tau}\left(\sigma\Psi_{\sigma\tau}+d\Psi_{d\tau}\right)(z\sigma,z',zd)\,dz'dz\right|
\\
\notag
&\lesssim ( |\sigma|+ |d|)\tau
\\
& \stackrel{\eqref{E:roijgr}}{\lesssim} \left( |\sigma|+\ell_{g} |\widehat\sigma|+\int_{(j-1)\eps_\nu}^{(j+1)\eps_\nu}|\alpha|\right)\cdot \tau
\label{E:singlejump}
\ea
since the derivatives $\partial_{\sigma\tau}|\sigma'-\sigma|$ and $\partial_{d\tau}|\sigma'-\sigma|$ are easily well defined for $\sigma\neq 0$ and locally bounded: notice that we differentiate only once the elementary curve of right states of the $k$-th family in its parameter and more times the strengths of the Riemann problem in the left / right states, thanks to the smoothness of the matrix $A$ in \eqref{eq:sysnc}.

Since~\eqref{E:singlejump} holds at each jump either of $\overline{\uu}(x)$ or $ \vv(x)$, then by algebraic computations we get the thesis.
\end{proof}

\begin{remark}
\label{R:timeupdate}
When $\overline{\ww}=\overline{\uu}$, then the proof of
Lemma~\ref{L:timeEst} states that where $\overline{\uu}$ has a jump
of strength $\sigma$ then the strength $\sigma'$ of the corresponding jump in
$\vv$ by~\eqref{E:singlejump} satisfies
\[
0< (1-\OL(1)\tau)\,\sigma\leq \sigma' \leq (1+\OL(1)\tau)\,\sigma
 \qquad
 \text{or}
 \qquad
  (1+\OL(1)\tau)\,\sigma\leq \sigma' \leq (1-\OL(1)\tau)\,\sigma <0\ .
\]
Moreover, if $\overline{u}$ does not any jump at $x=j\eps_\nu$, 
then the new jump introduced because of the discontinuity of $g_\nu$
at $j\eps_\nu$ satifies
$$
\sum_{k=1}^N|\sigma''_{k}|\leq \tau \int_{(j-1)\eps_\nu}^{(j+1)\eps_\nu}|\alpha|\,,
$$
where $\sigma''$ is the strenght of the new front of the $k$-th family
emerging from $(\tau, j\eps_\nu)$.
\end{remark}

We are now able to present the proof of Theorem~\ref{T:localConv} above.

\begin{proof}
We first prove by induction that if  
\be 
\label{E:rrepibm} 
 \dssb+G^{2}T\leq \dH
\ee
then estimate~\eqref{E:domt} concerning $\Upsilon\left(\ww_{\nu}(t+,\cdot)\right)$ holds when $0<t<T$.
We then prove the comparison with the exact viscous semigroup~\eqref{E:uniqEst}.
For brevity, we denote $\Upsilon(t):=\Upsilon\left(\ww_{\nu}(t+,\cdot)\right)$ all along this proof.

\firststep
\step{Initial step of induction}
We show, assuming~\eqref{E:rrepibm}, that
\bel{E:dadecidere1}
\Upsilon\left(\tau_{\nu}+\right)
\,\leq \,  \Upsilon\left(0\right)+G^{2}\tau_{\nu}
\,\stackrel{\eqref{E:rgabab}}{\leq} \,   \Upsilon\left(\overline u(\cdot)\right)+G^{2}\tau_{\nu} \ .
\ee
In particular, this step shows that if $\Upsilon\left(\overline u(\cdot)\right)\leq\dssb$ then being $\tau_{\nu}<T$ one has
\[
 \Upsilon\left(\tau_{\nu}+\right)
 \leq
 \dssb+G^{2}\tau_{\nu}
 \leq 
  \dssb+G^{2}T\leq
\dH \ .
\]
In particular one can restart the iteration procedure for defining $u_{\nu}$ up to $2\tau_{\nu}+$.
Observe first of all that estimate~\eqref{E:rgabab} allows to construct~\cite{AMfr} the wave-front-tracking approximation $\ftS{ t}{0}\overline\uu$ for all $t>0$.
By definition and by estimates~\eqref{E:rwrg}-\eqref{E:rgabab} recalled above then
\bas
\Upsilon( \tau_{\nu}+)=\Upsilon\left(\ww_{\nu}( \tau_{\nu}+,\cdot)\right)
&\equiv 
\Upsilon\left( \ftS{ \tau_{\nu}}{0}\overline\uu_{\nu}+ \tau_{\nu}g\left( \tau_{\nu},\cdot, \ftS{ \tau_{\nu}}{0}\overline\uu_{\nu}\right) \right)
\\
&\stackrel{\eqref{E:rwrg}}{\leq} 
\Upsilon\left( \ftS{ \tau_{\nu}}{0}\overline\uu_{\nu}\right) +G^{2}\tau_{\nu}
\\
&\stackrel{\eqref{E:rgabab}}{\leq}
\Upsilon\left( \overline\uu_{\nu}\right)+G^{2}\tau_{\nu}
\\
&\stackrel{\phantom{\eqref{E:rgabab}}}{\leq}
 \dssb +G^{2}\tau_{\nu}\ .
\eas

\step{Induction step I}
Suppose that $\Upsilon((n-1)\tau_{\nu}+)\leq \dH$.
We show that
\bel{E:dadecidere2}
\Upsilon\left(n\tau_{\nu}+\right)
\leq 
\Upsilon((n-1)\tau_{\nu}+)+G^{2}\tau_{\nu} \ .
\ee 
In this step we adopt the notation $\ww_{\nu}^{n-1}(\cdot)=\ww_{\nu}( (n-1)\tau_{\nu}+,\cdot)$ for the approximation at time $(n-1)\tau_{\nu}+$.
By definition and by the estimates~\eqref{E:rwrg}-\eqref{E:rgabab} recalled above one has
\bas
\Upsilon(n\tau_{\nu}+)
&\equiv 
\Upsilon\left( \ftS{n \tau_{\nu}}{(n-1) \tau_{\nu}}\ww_{\nu}^{n-1}+ \tau_{\nu}g\left( n \tau_{\nu},\cdot, \ftS{n \tau_{\nu}}{(n-1) \tau_{\nu}}\ww_{\nu}^{n-1}\right) \right)
\\
&\stackrel{\eqref{E:rwrg}}{\leq} 
\Upsilon\left( \ftS{n \tau_{\nu}}{(n-1) \tau_{\nu}}\ww_{\nu}^{n-1}\right)
+G^{2}\tau_{\nu}
\\
&\stackrel{\eqref{E:rgabab}}{\leq}
\Upsilon\left( \ww_{\nu}^{n-1}\right)+ 
G^{2}\tau_{\nu}
=
  \Upsilon((n-1)\tau_{\nu}+) +G^{2}\tau_{\nu}\ .
\eas

\step{Conclusion of~\eqref{E:domt}}
We deduce that whenever~\eqref{E:rrepibm} holds then
\bel{E:dadecidere3}
\Upsilon\left(t+\right)
\leq 
 \dssb+G^{2}t \leq \dH
\qquad
\text{for $0<t<T$.}
\ee
In particular, we show that $\ww_{\nu}( k\tau_{\nu}+,\cdot)$ is well defined for all $0\leq k\tau_{\nu}<T$.

By~\eqref{E:rgabab} and the definition of the approximation, it suffices to prove~\eqref{E:dadecidere3} at time-steps.
Estimate~\eqref{E:dadecidere1} provides the thesis at the first time-step $t=\tau_{\nu}$. At later time-steps, the thesis follows by induction by~\eqref{E:dadecidere2}.

\step{Proof of~\eqref{E:uniqEst}}
We recall~{\cite{Ch1,Ch2}} that there exists a small enough $\bar s>0$ for which one has the estimate
\bel{E:estVisc}
\forall s\leq \bar s, \forall \overline{\ww}\in \mathfrak D, \forall 0\leq h\leq t-s
\qquad
\lVert \vSC{h+s}{h}\overline{\ww}-\vSB{h+s}{h}\overline{\ww}-sg(h,\cdot,\overline{\ww}) \rVert_{L^{1}} \leq \OL(1)s^{2} \ .
\ee
By the triangular inequality
\[
\begin{split}\lVert g(h,\cdot,\overline{\ww} )-g(h+s,\cdot,\vSC{h+s}{h}\overline{\ww})  \rVert_{L^{1}} 
\leq
&\lVert g(h,\cdot,\overline{\ww} )-g(h+s,\cdot, \overline{\ww} ) \rVert_{L^{1}} 
\\&\qquad+
\lVert g(h+s,\cdot, \overline{\ww} ) -g(h,\cdot,\vSC{h+s}{h}\overline{\ww} ) \rVert_{L^{1}} 
\end{split}
\]
Assumption (G) at Page~\pageref{Ass:G} thus yields that for $s\to0$, denoting by $o(s)$ a function such that $o(s)/s\to0$,
\bel{E:estVisc2}
\forall s\leq \bar s, \forall \overline{\ww}\in \mathfrak D, \forall 0\leq h\leq t-s
\qquad
\lVert \vSC{h+s}{h}\overline{\ww} -\vSB{h+s}{h}\overline{\ww} -sg(h+s,\cdot, \vSC{h+s}{h}\overline{\ww} ) \rVert_{L^{1}} \leq Co(s) \ .
\ee
Let's adopt the shortcut $\ww_{\nu}^{n}(\cdot)$ for $\ww_{\nu}(n\tau_{\nu}+,\cdot)$: then by~\eqref{E:approx} and the triangular inequality
\bas
\lVert &\ww_{\nu}^{n} - \vSC{n\tau_{\nu}}{(n-1)\tau_{\nu}} \ww_{\nu}^{n-1} \rVert_{L^{1}} 
\\
&\equiv
\lVert \ftS{n \tau_{\nu}}{(n-1) \tau_{\nu}} \ww_{\nu}^{n-1}  
+ \tau_{\nu} g\left( n \tau_{\nu},\cdot, \ftS{n \tau_{\nu}}{(n-1)\tau_{\nu}} \ww_{\nu}^{n-1} \right) 
- \vSC{n\tau_{\nu}}{(n-1)\tau_{\nu}} \ww_{\nu}^{n-1} \rVert_{L^{1}} 
\\
&\leq
\lVert \vSB{n \tau_{\nu}}{(n-1) \tau_{\nu}} \ww_{\nu}^{n-1}  + \tau_{\nu} g\left( n \tau_{\nu},\cdot, \ftS{n \tau_{\nu}}{(n-1)\tau_{\nu}} \ww_{\nu}^{n-1} \right) 
- \vSC{n\tau_{\nu}}{(n-1)\tau_{\nu}} \ww_{\nu}^{n-1} \rVert_{L^{1}} 
\\
&\qquad +\lVert \ftS{n \tau_{\nu}}{(n-1) \tau_{\nu}} \ww_{\nu}^{n-1} -\vSB{n \tau_{\nu}}{(n-1) \tau_{\nu}} \ww_{\nu}^{n-1}   \rVert_{L^{1}} 
\eas
We directly estimate the first addend by~\eqref{E:estVisc2}, the second addend by~\eqref{E:ftest1}:
\bas
\lVert &\ww_{\nu}^{n} - \vSC{n\tau_{\nu}}{(n-1)\tau_{\nu}} \ww_{\nu}^{n-1} \rVert_{L^{1}} 
\leq \OL(1) \left(o(\tau_{\nu})+\varepsilon_{\nu} \tau_{\nu} \right)
\eas
The proof of~\eqref{E:uniqEst} is thus concluded.
\end{proof}

\subsection{Converge of time-step approximations to the viscous solution}

\begin{theorem}
\label{T:localEst}
Suppose there exists $0<\dss<\dH$, $ T>0$ and a closed domain
\[
\mathfrak D:=\mathfrak D(\dss):=\left\{\uu\in L ^{1}(\R;\R^{N})\cap \BV(\R;\R^{N})\ :\  \Upsilon(\uu)\leq \dss\right\}
\]
such that 
\begin{itemize}
\item for all $\nu$ and for every piecewise-constant initial data $\overline{\uu}\in\mathfrak D$ the approximations $\ww_{\nu}$ constructed in \S~\ref{sec:PCA} satisfy estimates~\eqref{E:domt}-\eqref{E:uniqEst} in $[0,T]$ and 
\item for all $\overline{\uu}\in\mathfrak D$ and for $0\leq t\leq T$~\cite{Ch1} provides a vanishing viscosity solution $\vSC{t}{0}\overline{\uu}$ of the Cauchy problem~\eqref{eq:sysnc}-\eqref{eq:inda}.
\end{itemize}
Then for every $\overline{\uu}\in\mathfrak D$ one can choose a suitable piecewise-constant approximation $\overline{\ww}_{\nu}\in \mathfrak D$ of $\overline{\uu}$ such that, denoting by $\ww_{\nu}$ the $\nu$-approximation as in \S~\ref{sec:PCA} with initial datum $\overline{\ww}_{\nu}$, for a.e.~$t\in[0,T]$ the sequence $\ww_{\nu}(t,\cdot)$ converges in $L^{1} (\R;\R^{N})$ to $\vSC{t}{0}\overline{\ww}$.
\end{theorem} 
\begin{proof}
\firststep\step{Introduction}
Let $\vSB{t}{h} w$ denote the semigroup of the homogeneous system constructed by vanishing viscosity~\cite{BB} where the `initial datum' is fixed at time $h$ rather than at $h=0$.
We recall~{\color{red}\cite{Ch1,Ch2}} that there exists $L>0$ s.t.~for $\overline{\ww}_{1},\overline{\ww}_{2}\in\mathfrak D$, $h\in[0,T]$, $t_{1},t_{2}\in[h,T]$ then
\bel{E:lipG}
\lVert \vSC{t_{1}}{h}\overline{\ww}_{1}(\cdot)-\vSC{t_{2}}{h}\overline{\ww}_{2}(\cdot)\lVert_{L^{1}} 
\leq 
L \left( \lVert  \overline{\ww}_{1}-  \overline{\ww}_{2}\lVert_{L^{1}} + |t_{2}-t_{1}| \right)
\ee
and for $\overline w\in\mathfrak D$, $0\leq t_{0}\leq  t_{1}\leq  t_{2}\leq T$, one has the semigroup property
\bel{E:sem}
\vSC{t_{2}}{t_{0}}\overline{\ww}\equiv \vSC{t_{2}}{t_{1}}\left[\vSC{t_{1}}{t_{0}}\overline{\ww}\right]
\ee

\step{Strategy} In the spirit of~{\cite[Theorem 2.9]{BressanBook}}, fix any $0\leq \bar t \leq T$ and define the auxiliary function
\[
\Psi_{\nu}(t,\cdot)= \vSC{\bar t}{t}\left[\ww_{\nu}(t+,\cdot)\right]-\vSC{\bar t}{0}\overline{\ww}_{\nu}(\cdot)
\qquad t\in [0,\bar t]
\]
We now prove that $\ww_{\nu}(\bar t, \cdot)$ converges in $L^{1}$ to $\vSC{\overline t}{0}\overline{\ww}$ by showing that the following limit vanishes:
\bas
\lim_{\nu}\lVert \ww_{\nu}(\bar t,\cdot)- \vSC{\bar t}{0}\overline{\ww}(\cdot)\rVert_{L^{1}}
&\leq 
\lim_{\nu}\lVert \ww_{\nu}(\bar t,\cdot)- \vSC{\bar t}{0}\overline{\ww}_{\nu}(\cdot)\rVert_{L^{1}} +\lim_{\nu}\lVert  \vSC{\bar t}{0}\overline{\ww}(\cdot)- \vSC{\bar t}{0}\overline{\ww}_{\nu}(\cdot)\rVert_{L^{1}}
\\
&\equiv \lim_{\nu}\lVert \Psi_{\nu}(\bar t,\cdot)\rVert_{L^{1}} +0 \ .
\eas
The second addend indeed is trivially converging to $0$ by~\eqref{E:lipG} as $\overline{\uu}_{\nu}$ converges to $\overline{\ww}$ in $L^{1}$.

\step{Estimates}
Let $\bar t=n\tau_{\nu}+\hat t$ with $n\in\nat\cup\{0\}$ and $\hat t\in[0,\tau_{\nu})$.
By the triangular inequality
\ba
\lVert \Psi_{\nu}(\bar t,\cdot) \rVert_{L^{1}}
&\equiv\lVert \Psi_{\nu}(\bar t,\cdot) - \Psi_{\nu}(0,\cdot) \rVert_{L^{1}}
\notag
\\
&\leq 
\sum_{k=0}^{n-1} \lVert \Psi_{\nu}((k+1)\tau_{\nu},\cdot) - \Psi_{\nu}(k\tau_{\nu},\cdot) \rVert_{L^{1}}
+\lVert \Psi_{\nu}( \bar t,\cdot) - \Psi_{\nu}(n\tau_{\nu},\cdot) \rVert_{L^{1}}
\notag
\\
&\leq
\sum_{k=0}^{n-1} \lVert \Psi_{\nu}((k+1)\tau_{\nu},\cdot) - \Psi_{\nu}(k\tau_{\nu},\cdot) \rVert_{L^{1}}
+ \OL(1)\tau_{\nu} \ .
\label{E:gwagb}
\ea
In the last step we estimated the norm of
$\Psi_{\nu}( \bar t,\cdot) - \Psi_{\nu}(n\tau_{\nu} ,\cdot) \equiv  \ww_{\nu}(\bar t+,\cdot) - \vSC{\bar t }{n\tau_{\nu}}\left[\ww_{\nu}(n\tau_{\nu} +,\cdot)\right]$ by the Lipschitz continuity~\eqref{E:estVisc}-\eqref{E:ftest1} exploiting the fact that $n\tau_{\nu}\leq\bar t<(n+1)\tau_{\nu}$ and thus by definition $\ww_{\nu}(\bar t,\cdot) \equiv \ftS{\bar t}{n\tau_{\nu}}\ww_{\nu}(n\tau_{\nu} +,\cdot)$.
Moreover, by definition and the semigroup property~\eqref{E:sem}
\bas 
 \lVert& \Psi_{\nu}((k+1)\tau_{\nu},\cdot) - \Psi_{\nu}(k\tau_{\nu},\cdot) \rVert_{L^{1}}
\equiv
 \lVert \vSC{\bar t}{ (k+1)\tau_{\nu}}\left[\ww_{\nu}((k+1)\tau_{\nu}+,\cdot)\right]-\vSC{\bar t}{ k\tau_{\nu}}\left[\ww_{\nu}( k\tau_{\nu}+,\cdot)\right] \rVert_{L^{1}}
 \\
&\equiv
\lVert \vSC{\bar t}{ (k+1)\tau_{\nu}}\left[\ww_{\nu}((k+1)\tau_{\nu}+,\cdot)\right]
-\vSC{\bar t}{(k+1)\tau_{\nu}}\left[\vSC{(k+1)\tau_{\nu}}{ k\tau_{\nu}}\left[\ww_{\nu}( k\tau_{\nu}+,\cdot)\right]\right] \rVert_{L^{1}}
\\
&\stackrel{\eqref{E:lipG}}{\leq}
L\lVert \ww_{\nu}((k+1)\tau_{\nu}+,\cdot)- \vSC{(k+1)\tau_{\nu}}{ k\tau_{\nu}}\left[\ww_{\nu}( k\tau_{\nu}+,\cdot)\right]  \rVert_{L^{1}} \ .
\eas
Estimating this term by~\eqref{E:uniqEst} and plugging this into~\eqref{E:gwagb} we finally deduce, being $n\tau_\nu=T$, that
\bas 
\lVert \Psi_{\nu}(\bar t) \rVert_{L^{1}}
\leq
\OL(1)\tau_{\nu}+\sum_{k=0}^{n-1}\OL(1) \left(o(\tau_{\nu})+\eps_{\nu}\tau_{\nu} \right)
\leq
\OL(1)\tau_{\nu} + T \left(\frac{o(\tau_{\nu})}{\tau_{\nu}}  +\eps_{\nu}   \right) \xrightarrow{\nu\uparrow\infty}0 \ .
\eas
This concludes the proof of the $L^{1}$-convergence.
\end{proof}

\section{Qualitative properties of the entropy solution}
\label{S:qualitative}

\newcommand{\oc}{\overline{c}}
\newcommand{\op}{\overline{p}}
\newcommand{\oq}{\overline{q}}
\newcommand{\og}{\overline{\gamma}}
\newcommand{\ox}{\overline{x}}
\newcommand{\ot}{\overline{t}}
\newcommand{\oP}{\overline{P}}

This section is devoted to the proof of Theorem~\ref{Th:structure} concerning the structure of solutions to balance laws when the characteristic fields are 
\begin{itemize}
\item either linearly degenerate in the sense of Definition~\ref{D:LD},
\item or piecewise-genuinely nonlinear in the sense of Definition~\ref{D:pwGN}.
\end{itemize}
We work under the standard Lipschitz regularity assumption \textbf{(G)} at Page~\pageref{Ass:G} on the source term, and we assume furthermore in this section that $g$ only depends on the state variable:
\[
g=g(u) \ .
\]

The proof is by approximation, following ideas already in~\cite{BreLF,BressanBook,BYu}. 
We construct suitable objects, estimates and arguments on the approximate solutions defined in \S~\ref{sec:PCA}.
Owing to the convergence result proved in \S~\ref{S:convergenceAndExistence}, we are then able to obtain our thesis in the limit.
The section is organized as follows:
\begin{itemize}
\item[\S~\ref{S:balances}] Establishes balances for the positive/negative amount of $i$-waves in a space-time region.
\item[\S~\ref{S:discontinuities}] Defines sub-discontinuities of shocks for each piecewise genuinely nonlinear families.
\item[\S~\ref{Ss:constructionlimitcurves}] Defines the fractional-step approximations of $i$-shocks and $i$-contact discontinuities.
\item[\S~\ref{Ss:pwgnflimit}] Proves the limits in Theorem~\ref{Th:structure} at shocks of piecewise genuinely nonlinear families.
\item[\S~\ref{Ss:rfroionrgnjg}] Proves the limits in Theorem~\ref{Th:structure} at contact discontinuities of linearly degenerate families.
\item[\S~\ref{Ss:continuityPoints}] proves the limits in Theorem~\ref{Th:structure} at continuity points.
\item[\S~\ref{Ss:geomLemmas}] Contains elementary geometric lemmas on piecewise genuinely nonlinear families.
\item[\S~\ref{Ss:auxiliary}] Contains the proof of intuitive auxiliary lemmas.
\end{itemize}

\subsection{Preliminary estimates: balances on characteristic regions}
\label{S:balances}

In this section we generalize balances for the flux of positive and negative waves of a fixed approximation $u_{\nu}$ which was constructed in \S\S~\ref{sec:PCA}-\ref{S:convergenceAndExistence}. These balances reduce to well known ones for the homogeneous system which are for example in~\cite[\S~7.6]{BressanBook}.

We remind~\cite[\S~7.6]{BressanBook} in particular the definition of the interaction measure $\mu_{\nu}^{I}$ and interaction-cancellation measure $\mu_{\nu}^{IC}$: they are purely atomic measures which are concentrated at interaction points of physical fronts belonging to two characteristic families $i$, $j$. If  $\sigma'$ and $\sigma''$ are the incoming strengths of the fronts interacting at a point $P$ then, using the Definition~\ref{def:qi} of amount of interaction, one has
\begin{subequations}
\label{E:muIC}
\ba
&\mu_{\nu}^{I}(\{P\}):={\I}(\sigma^\prime,\,\sigma^{\second}) 
\\
&\mu_{\nu}^{IC}(\{P\}):={\I}(\sigma^\prime,\,\sigma^{\second}) +\begin{cases}  |\sigma'|+|\sigma''|-|\sigma'+\sigma''|  &\text{if $i=j$,}\\ 0 & \text{if $i\neq j$.} \end{cases}
\ea
\end{subequations}
Of course $0\leq \mu_{\nu}^{I}\leq \mu_{\nu}^{IC}$.
We state that the interaction-cancellation measure can be controlled by
$\Q$ even when a source term is present.
\begin{lemma} 
The interaction-cancellation measure satisfies the estimates:
\[
 \mu_{\nu}^{IC}((t_{1},t_{2}]\times\R)
\lesssim
 \TV^{-}\left( \Q(u_{\nu});(t_{1},t_{2}]\right)
 \qquad
\forall \ 0\leq t_{1}< t_{2}\ ,
\]
where $ \TV^{-}$ is the negative total variation, and $\Q$ is the interaction
potential defined at~\eqref{eq:ip}.
In particular, $\mu_{\nu}^{IC}$ is a locally bounded Radon measure.
\end{lemma}
\begin{proof}
At an interaction point $P=(t,x)$ by classical interaction estimates, as recalled in Proposition~\ref{pro:ie}
\[
0<\mu_{\nu}^{IC}(P) \lesssim
 \left| \Q(u_{\nu} (t -))-\Q(u_{\nu}(t +))\right| \ .
\] 
At time updates $k\tau_{\nu}$ the interaction cancellation measure is null by construction, thus the thesis holds trivially.
We also recall that we stay in a domain with small total variation.
\end{proof}

%
Consider a polygonal region $\Gamma$ with edges transversal to the
waves it encounters.
Consider the total amount $W^{\nu i\pm}_{\rin}$, $W^{\nu i\pm}_{\out}$ of
positive and negative $i$-waves entering the region:
\[
W^{\nu i\pm}_{\rin}(\Gamma):=\sum_{\text{entering $\Gamma$} }s_{i}^{\pm} \ ,
\quad
W^{\nu i\pm}_{\out}(\Gamma):=\sum_{\text{exiting $\Gamma$} }s_{i}^{\pm}\ ,
\quad
s_{i}^{\pm}=\max\{\pm s_{i},0\} \ .
\] 
Define the incoming and outgoing flux of the $i$-th wave through the
boundary of the region as
\[
W^{\nu i}_{\rin}=W^{\nu i+}_{\rin}-W^{\nu i-}_{\rin},
\qquad
W^{\nu i}_{\out}=W^{\nu i+}_{\out}-W^{ \nu i-}_{\out}
\qquad i=1,\dots,N.
\]
\begin{lemma}
\label{L:rqgsbb}
There exists a positive constant {$C$ depending only on $A$, $g$, $\dssb$, $T$} such that the following holds.
If $\Gamma\subset[h\tau_{\nu},k\tau_{\nu}]\times\R$, for some $h,k\in\{0,1,\dots,\mathrm{floor}( {T}/{\tau_\nu})\}$, then for $i=1,\dots,N$ one has the estimate
\begin{subequations}
\label{E:bal}
\ba
e^{-(k-h)\tau_{\nu}C}
\left[
W^{\nu i+}_{\rin}-C\mu_{\nu}^{IC}\left(\overline\Gamma\right) 
\right]
\leq & W^{\nu i+}_{\out}
\leq  
e^{(k-h)\tau_{\nu}C}
\left[
W^{\nu i+}_{\rin}+C\mu_{\nu}^{IC}\left(\overline\Gamma\right) 
\right] \ ,
\\
e^{-(k-h)\tau_{\nu}C}
\left[
W^{\nu i-}_{\rin}-C\mu_{\nu}^{IC}\left(\overline\Gamma\right) 
\right]
\leq & W^{\nu i-}_{\out}
\leq  
e^{(k-h)\tau_{\nu}C}
\left[
W^{\nu i-}_{\rin}+C\mu_{\nu}^{IC}\left(\overline\Gamma\right) 
\right]
\ .
\ea
\end{subequations}
\end{lemma}
\begin{proof}
Waves might change only at interaction times and at update times. Denote by
\[
 W^{\nu i+}(t),\quad W^{\nu i-}(t )
 \]
respectively the positive and negative $i$-waves of $u_{\nu}$ present in
$\Gamma$ at time $t$.
\firststep
\step{Interaction times $t$} Denote by $\sigma'$ and $\sigma''$ the incoming
strengths of the fronts interacting at a point $P\in\Gamma$. By interaction
estimates~\cite[Lemma~1]{AMfr}, as in~\cite[(7.98)]{BressanBook}, one has
\[
|W^{\nu i+}(t+)-W^{\nu i+}(t-)|+|W^{\nu i-}(t+)-W^{\nu i-}(t-)|
\lesssim \mu_{\nu}^{IC}(\{P\}) \ .
\]
\step{Update time $t$}
We can assume that no front enters / exits $\Gamma$ at the update time.
Since we are considering $g=g(u)$, then denoting by $\tau_{\nu}$ the time-step,
by Remark~\ref{R:timeupdate} in the construction of the approximation one has
precisely as in~\cite[(2.8)]{AG2} that
\[
\left| W^{\nu i\pm}(t+) -W^{\nu i\pm}(t-)  \right| \lesssim \tau_{\nu}W^{\nu i\pm}(t-)  \ .
\]
In particular, if in the interval $[t_{1},t_{2})$ there is no interaction and no wave enters / exits $\Gamma$, then
\[
e^{-C(k-h)\tau_{\nu}}  W^{\nu i\pm}(t_{1}-) \leq W^{\nu i\pm}(t_{2}-)  \leq  e^{C(k-h)\tau_{\nu}}  W^{\nu i\pm}(t_{1}-)  
\] 
where $h\tau_{\nu} \leq t_{1}<t_{2}\leq k\tau_{\nu}$ for some $h,k\in\nat$.

\step{Conclusion}
Combining in a rough way the estimates in the previous steps, and since
taking into account that waves might enter later than $h\tau_{\nu}$ or might
exit before $k\tau_{\nu}$ the estimate would just be finer, we get the thesis by
standard calculus.
\end{proof}

\subsection{Definition of approximate sub-discontinuity curves}
\label{S:discontinuities}

Assume that the $i$th-characteristic field is piecewise genuinely nonlinear
as in Definition~\ref{D:pwGN}.
Let us directly assume that $\omega_{i}^{1}[u^{-}],\dots,
\omega_{i}^{J_{i}}[u^{-}]$, defined as follows, are monotone increasing:
\[
\omega_{i}^{j}[u^{-}]\ :\ R_{i}[u^{-}](\omega_{i}^{j}[u^{-}]) \in Z_{i}^{j},
\quad  j=1,\dots,J_{i},\qquad
\omega_{i}^{0}[u^{-}]:=-\infty,\ \omega_{i}^{J_{i}+1}[u^{-}]:=+\infty
\]
where we remind that the hyper-surfaces $Z_{i}^{j}$ are the connected components of
\[
Z_{i}=\{u : \nabla\lambda_{i}(u)\cdot r_{i}(u)=0\}
=\bigcup_{j=1}^{J_{i} }Z_{i}^{j}\ .
\]
We directly assume, since the analysis of the other case is perfectly analogous, that
\[
\begin{cases}
\nabla\lambda_{i}(u)\cdot r_{i}(u)<0&\text{if $j$ is even and
$\omega_{i}^{j}[u]<0<\omega_{i}^{j+1}[u]$,} 
\\
\nabla\lambda_{i}(u)\cdot r_{i}(u)>0&\text{if $j$ is odd and
$\omega_{i}^{j}[u]<0<\omega_{i}^{j+1}[u]$}. 
\end{cases}
\] 
\begin{figure}\centering
\begin{tikzpicture}[xscale=2]
      \FPeval{\result}{round(2^0.5,9)}%
      \draw[-stealth,thick] (-3.2,0) -- (3.2,0) node[below] {$u$};
      \draw[thick,dashed] (-\result,-1.7) -- (-\result,1.7) ;
      \draw (-\result,-0.4) node[right] {$R[u^{L}](w^{1})$};
      \draw[OliveGreen] (-\result,1.7) node[left] {$Z^{1}$};
    \draw[purple](-\result, -1.7)  node[below] { $u^{1}$};
      \draw[ thick,dashed] (0,-1.7) -- (0,1.7) ;
       \draw (0,-0.4) node[right] {$R[u^{L}](w^{2})$};
       \draw[OliveGreen] (0,1.7) node[left] {$Z^{2}$};
    \draw[purple](0, -1.7)  node[below] { $u^{2}$};
     \draw[ thick,dashed] (\result,-1.7) -- (\result,1.7) ;
       \draw (\result,-0.3) node[right] {$R[u^{L}](w^{3})$};
       \draw[OliveGreen] (\result,1.7) node[left] {$Z^{3}$};
   \draw[purple](\result, -1.7)  node[below] { $u^{3}$};
      \draw[domain=-2.8:2.8,smooth,variable=\x,thick,blue] plot ({\x},{\x^5/20-\x^3/3});
     \draw[purple,thick] (-2, -32/20+8/3) -- (2.582,0) node[below right] {$ $};
     \draw[dashed,purple](-2, -32/20+8/3)--(-2, -1.7)  node[below] {$u^{R}=u^{0}$};
     \draw[dashed,purple](2.582, 0)--(2.582, -1.7)  node[below] {$u^{4}=u^{L}$};
    \draw[thick,purple](-\result, -1.7)--(0, -1.7)  ;
    \draw[thick,purple](\result, -1.7)--(2.582, -1.7)  ;
   \end{tikzpicture}
   \caption{We consider a jump $[u^{L},u^{R}]$ where $u^{R}=T_{i}[u^{L}](s)$ with $s<0$. We highlight different points belonging to the hyper-surfaces $Z_{i}^{1}$, $Z_{i}^{2}$ and $Z_{i}^{3}$: 
$\bullet$ the points $R_{i}[u^{L}](w^{1})$, $R_{i}[u^{L}](w^{2})$, $R_{i}[u^{L}](w^{3})$ of intersection with the rarefaction curve through $u^{L}$ and
$\bullet$ the points $u^{1} $,  $u^{2} $, $u^{3} $ in~\eqref{E:tauZ} which identify the $(i,3)$ and $(i,1)$ sub-discontinuity fronts $[u^{4},u^{3}]$, $[u^{2},u^{1}]$.
   }%
\end{figure}
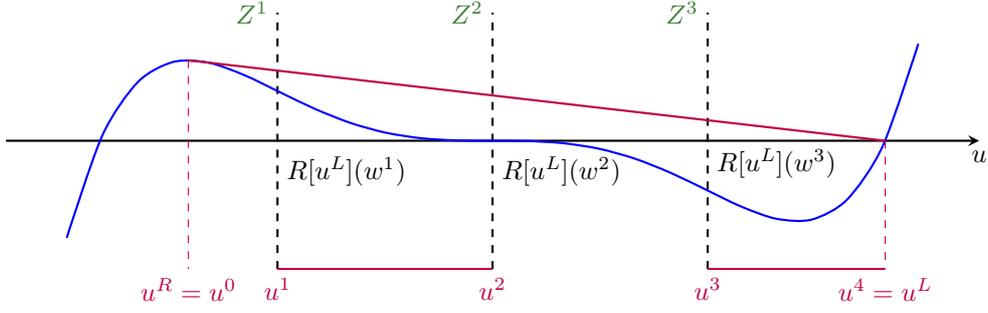
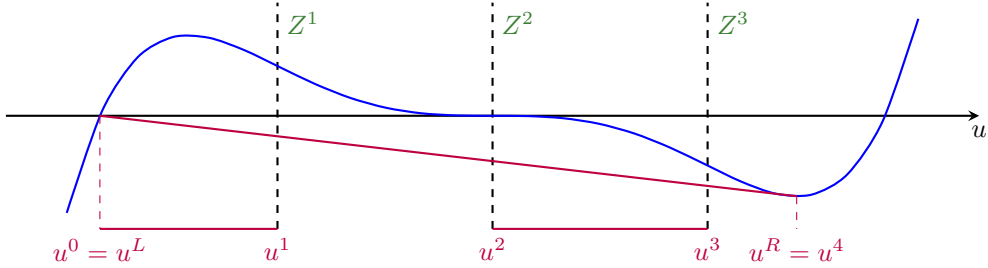
\begin{figure}\centering
\begin{tikzpicture}[xscale=2]
      \FPeval{\result}{round(2^0.5,9)}%
      \draw[-stealth,thick] (-3.2,0) -- (3.2,0) node[below] {$u$};
      \draw[thick,dashed] (-\result,-1.5) -- (-\result,1.5) ;
      \draw[OliveGreen] (-\result,1.5) node[below right] {$Z^{1}$};
    \draw[purple](-\result, -1.5)  node[below] { $u^{1}$};
      \draw[ thick,dashed] (0,-1.5) -- (0,1.5) ;
       \draw[OliveGreen] (0,1.5) node[below right] {$Z^{2}$};
    \draw[purple](0, -1.5)  node[below] { $u^{2}$};
     \draw[ thick,dashed] (\result,-1.5) -- (\result,1.5) ;
       \draw[OliveGreen] (\result,1.5) node[below right] {$Z^{3}$};
   \draw[purple](\result, -1.5)  node[below] { $u^{3}$};
      \draw[domain=-2.8:2.8,smooth,variable=\x,thick,blue] plot ({\x},{\x^5/20-\x^3/3});
     \draw[purple,thick] (2, 32/20-8/3) -- (-2.582,0) node[below right] {$ $};
     \draw[dashed,purple](2, 32/20-8/3)--(2, -1.5)  node[below] {$u^{R}=u^{4}$};
     \draw[dashed,purple](-2.582, 0)--(-2.582, -1.5)  node[below] {$u^{0}=u^{L}$};
    \draw[thick,purple](\result, -1.5)--(0, -1.5)  ;
    \draw[thick,purple](-\result, -1.5)--(-2.582, -1.5)  ;
   \end{tikzpicture}
   \caption{We consider a jump $[u^{L},u^{R}]$ where $u^{R}=T_{i}[u^{L}](s)$ with $s>0$. The $(i,0)$ and $(i,2)$ sub-discontinuities in this case are $[u^{0},u^{1}]$ and $[u^{2},u^{3}]$, where $u^{1} $,  $u^{2} $, $u^{3} $ are defined in~\eqref{E:tauZ}.
   }%
\end{figure}
For these piecewise genuinely nonlinear fields, following~\cite{BreLF, BressanBook, tplrpnpn,BYu} we now define approximate sub-discontinuities of a fixed approximation $u_{\nu}$ which was constructed in \S\S~\ref{sec:PCA}-\ref{S:convergenceAndExistence}.
In particular, we extend~\cite{BYu} in the presence of a Lipschitz continuous source term $g=g(u)$.

Let $[u^{L},u^{R}]$ be a wavefront of $u_{\nu}$ belonging to the $i$-th family, where $u^{R}=T_{i}[u^{L}](s)$.
Suppose for instance that $s>0$, which means $u_{i}^{R}>u_{i}^{L}$.
When the $i$-th field is piecewise genuinely nonlinear one can split $[u^{L},u^{R}]$ into sub-discontinuities: if $\bar u$ is the function defined in~\eqref{E:Rfp} for the construction of the Rieman solver, then since $\tau\mapsto\overline u\left( \tau ;u^{L},s\right)$ is transversal to $Z^{j}_{i}$ there are
\bel{E:tauZ}
0\leq \tau^{j_{1}}< \dots < \tau^{j_{2}}\leq s\quad:\quad u^{j_{1}+k}:= \overline u\left( \tau^{j_{1}+k};u^{L},s\right) \in Z^{j_{1}+k}_{i}\qquad k=0,\dots, j_{2}-j_{1}.
\ee
When $u^{L}$, $u^{R}$ do not belong to any $Z_{i}^{j}$ we still need to include the extremal points: set
\begin{itemize}
\item $\tau^{j_{1}-1}=0$ and $u^{j_{1}-1}=u^{L}$ in case $\tau^{j_{1}}>0$; 
\item $\tau^{j_{2}+1}=s$ and $u^{j_{2}+1}=u^{R}$ in case $\tau^{j_{2}}<s$.
\end{itemize}
If instead $s<0$ the definition is analogous with $0\geq \tau^{j_{2}}> \dots > \tau^{j_{1}}\geq s$, where $j_{2}>\dots>j_{1}$.

\begin{definition}
\label{D:subdisc1}
Suppose that $s^{j}_{i}=\tau^{j+1}-\tau^{j}\neq0$: then we call an $(i,j)$-sub-discontinuity of strength $s^{j}_{i}$ of the $i$-th wavefront $[u^{L},u^{R}]$ of $u_{\nu}$
\begin{itemize}
\item $[u^{j},u^{j+1}]$, if $s>0$ and $j$ is even, or
\item  $[u^{j+1},u^{j}]$, if $s<0$ and $j$ is odd.
\end{itemize}
\end{definition}

Notice that, by definition, the state vector of $(i,j)$-sub-discontinuities belongs to the part of the wavefront where the $i$-th eigenvalue is decreasing.
Rarefaction fronts are instead contained in regions where the $i$-th eigenvalue is increasing across the discontinuities.

One of the reasons to introduce sub-discontinuities when the flux is not genuinely nonlinear, but only piecewise genuinely nonlinear, is that discontinuities might split either at interaction times or at update times. Since the approximate solution of a Riemann problem contains at most one sub discontinuity $s_{i}^{j}$ for $j\in\{0,1,\dots, J_{i}\}$~\cite[Lemma~4.3]{BYu}, sub-discontinuities do not.

The next step is to identify which sub-discontinuities in the approximation $u_{\nu}$ are in the limit converging to a sub-discontinuity of the entropy solution $u$: we call these `surviving' discontinuities ``approximate discontinuities''. We fix for this purpose thresholds $\beta$ and $\beta/2$.

\begin{definition}
\label{D:subdisc}
Let $\beta>0$.
A maximal, leftmost $(\beta,i,j)$-approximate sub-discontinuity curve is any maximal (concerning set inclusion) closed polygonal line---parametrized with time in the $(t,x)$-plane---with nodes $(t_{0},x_{0})$, $(t_{1},x_{1})$, $\dots$, $(t_{n},x_{n})$, where $t_{0}\leq\dots\leq t_{n}$, such that
\begin{enumerate}
\item each node $(t_{k},x_{k})$, $k=1,\dots,n$ is an interaction point or an update time;
\item
\label{definitionSF2}
the segment $[(t_{k-1},x_{k-1}),(t_{k},x_{k})]$ is the support of an $(i,j)$-sub-discontinuity front with strength $|s_{i}^{j}|\geq\beta/4$ and there is at least one time $t\in[t_{0},t_{n}]$ such that $|s_{i}^{j}|\geq\beta$; the index ${j}$ must be either odd if the strength of the $i$-th jump $s_{i}>0$ or ${j}$ must be even if $s_{i}<0$;
\item 
\label{definitionSF3}
it stays on the left of any other polygonal line it intersects and having the above properties.
\end{enumerate}
\end{definition}

We write an interaction estimate for sub-discontinuities in order to familiarize with them.

\begin{lemma}
\label{E:knglnj}
For any compact $K\subset\Omega\subseteq\real^N$ there exist constants $C_{1}, C_{2}$ and $\chi_{1}$ so that:
Consider an interaction between a $i$-front strength $|s_{i}|$ and a $j$-front of strength $|s_{j}|$ for $i\neq j$ and $i,j\in\{1,\dots,N\}$.
Let $u^{L}$ / $u^{R}$ denote the left / right states of $u_\nu$ at that interaction, which belong to $K$, and let
$s_{1}^{+},\dots, s_{N}^{+}$ be the outgoing strengths, so that
\[
u^{R}=T_{N}[T_{N-1}[\dots T_{2}[T_{1}[u^{L}](s_{1}^{+})]
(s_{2}^{+})\dots](s_{N-1}^{+})](s_{N}^{+})\,.
\]
Then, calling $\oP$ the point of interaction, there holds
\[
|s_{i}^{+}-s_{i}|+|s_{j}^{+}-s_{j}|+\sum_{\ell\neq i,j}|s_{\ell}^{+}| \leq C_{1}\I(s_{i},s_{j})
\leq C_{1}|s_{i}s_{j}|= C_{1}\mu_{\nu}^{I}(\oP) \ .
\]
Moreover, consider at each node $\oP$ which is a point of interaction the strengths $s_{i}^{k-}$ / $s_{i}^{k+}$ of each incoming / outgoing
$(i,k)$-sub-discontinuity, possibly except for the first one $(t_{0},x_{0})$: they satisfy
\be
\label{E:grergrr}
|s_{i}^{k+}-s_{i}^{k-}|\leq \frac{C_{2}}{\beta}\mu^{I}_{\nu}(\oP) \ .
\ee
\end{lemma}

\begin{proof}
By classical interaction estimates, and the definition of the interaction measure, we just need to prove the last inequality concerning sub-discontinuities: the first part of the statement indeed is just by construction.
The proof of~\eqref{E:grergrr} is a consequence of the fact that the strength $s_{i}^{k}$ of any $(i,k)$ sub-discontinuity is Lipschitz continuous in the left and right states of the $i$-front, together with the estimates below.
If ${u'}^{L}$ and ${u'}^{R}$ are the left / right states of the $i$-front after the interaction then
\[
|{u'}^{L}-{u}^{L}|+|{u'}^{R}-{u'}^{R}|
\lesssim 
|s_{j}| \ .
\]
Since $|s_{i}|\geq|s_{i}^{k}|\geq\beta/4$ by Definition \ref{D:subdisc} of $(i,k)$-sub-discontinuity, and since $ \mu^{I}_{\nu}(\oP)= |s_{i}s_{j}|$ due to the fact that $i\neq j$, then $|s_{j}|\leq  \frac{|s_{i}|}{ \beta/4}|s_{j}|\leq \frac{4\mu^{I}_{\nu}(\oP)}{ \beta}$, from which we get the thesis.
\end{proof}

We are now able to determine a countable family $\J^{j}_{\beta,i}(\nu)$ of maximal, leftmost $(\beta,i,j)$-approximate sub-discontinuity curves of $u_{\nu}$ which will in the limit define the family of curves $\J$ in the statement of Theorem~\ref{Th:structure}.
This is due to the fact that when $\beta$ is fixed then the cardinality of maximal, leftmost $(\beta,i,j)$-approximate sub-discontinuity curves of $u_{\nu}$ is, definitively as $\nu\uparrow\infty$, bounded by a constant $M_{\beta,i,j}$ independent of $\nu$ thanks again to the bounds on the total variation---and of course up to a fixed finite time.
Notice that the set of curves $\J^{j}_{\beta,i}(\nu)$ enriches as $\beta\downarrow0$.
\begin{lemma}
\label{L:numFronts}
When the threshold $\beta$ is fixed, then the cardinality
$\sharp \J^{j}_{\beta,i}(\nu)=:M_{\beta,i,j}(\nu)$ of maximal, leftmost
$(\beta,i,j)$-approximate sub-discontinuity curves---up to any fixed positive
time---is uniformly bounded in $\nu$, and thus also in $i=1,\dots,N$ and $j$:
it is of order $M_{\beta,i,j}(\nu)\lesssim \beta^{-2}$.
\end{lemma}
\begin{proof}
Suppose that---up to the fixed time $\oT$ we are considering---the total variation
of every $u_{\nu}$ is less than $\oV$, which is possible by
Theorem~\ref{T:localConv}.

We begin fixing notations. Fix any admissible triple of indices $\beta$, $i$, $j$.
Consider a maximal, leftmost $(\beta,i,j)$-approximate sub-discontinuity
curve $\og$ of $u_{\nu}$ for $\nu$ large enough: denote 
\begin{itemize}
\item by $|s_{i}^{j}(t)|$ the
strength of the $(\beta,i,j)$-approximate sub-discontinuity $\og$ of $u_{\nu}$, 
\item by $\oP_{k}=(\ot_{k},\ox_{k})$ the nodes of the $(\beta,i,j)$-approximate sub-discontinuity $\og$ of $u_{\nu}$, and
\item by $|s_{i}(t)|$ the strength of the whole discontinuity $\og$ of $u_{\nu}$, for $\ot_{0}<t<\ot_{n}$.
\end{itemize}
Suppose $\beta<\frac{1}{1000}\mathrm{dist}( Z_{i}^{j}, Z_{i}^{j+1})$ and
$ \varepsilon_{\nu}\ll10\beta\ll1$: this is allowed since $ \varepsilon_{\nu}\downarrow0$ and since decreasing $\beta$ increases the number
of maximal, leftmost $(\beta,i,j)$-approximate sub-discontinuity curves.
Fix for example $s_{i}^{j}>0$ for notational convenience, the other case being similar.

Before proving Lemma \ref{L:numFronts}, we remind relevant estimates at nodes:
\begin{itemize}
\item
By Lemma~\ref{E:knglnj} if $\og$ interacts with an $i$-front of a different
characteristic family
\ba
\label{E:grgrwgewgregr}
|s_{i}^{j}(\ot_{k}+)-s_{i}^{j}(\ot_{k}-)|\leq \frac{C_{2}}{\beta}\mu^{IC}_{\nu}(\oP_{k}) \ .
\ea
\item
At any update time $\ot_{k}$, denoting by $C$ the constant given by Remark~\ref{R:timeupdate}, one has
\ba
\label{E:egqrgregergregre}
|s_{i}^{j}(\ot_{k}+)-s_{i}^{j}(\ot_{k}-)| 
< 
C \tau_{\nu }|s_{i}(\ot_{k}-)| < C \tau_{\nu }\oV \ . 
\ea
\item
Let $u^{L}$ / $u^{R}$ denote the left and right value of the maximal, leftmost $(\beta,i,j)$-approximate sub-discontinuity
curve $\og$ that we are considering.
We make preliminary observations concerning interactions among $i$-waves, since we are interested only in the strength of the $j$-th component, before providing complete estimates:
\begin{enumerate}
\item Suppose that $u^{R}\in Z_{i}^{j+1}$ and $u^{L}\in Z_{i}^{j} $ between $\ot_{k-1}$ and $\ot_{k}$. Then: $\bullet$ In case $\og$ interacts at $\oP_{k}$ with another $i$-shock then we have that $u^{R}\in Z_{i}^{j+1}$ and $u^{L}\in Z_{i}^{j} $ also for $\ot_{k}< t< \ot_{k+1}$. If the fixed $\beta$ is smaller than $\frac{1}{1000}\mathrm{dist}( Z_{i}^{j}, Z_{i}^{j+1})$, this proves that the strength of the $(\beta,i,j)$-approximate sub-discontinuity is more than $\beta$ at both times $\ot_{k}^{-}$ and $\ot_{k}^{+}$.
 $\bullet$ In case $\og$ interacts at $\oP_{k}$ with an $i$-rarefaction, since the strength of rarefactions is vanishingly small again the strength of the $(\beta,i,j)$-approximate sub-discontinuity is more than $\beta$ at both times $\ot_{k}^{-}$ and $\ot_{k}^{+}$ owing to the condition $\beta<\frac{1}{1000}\mathrm{dist}( Z_{i}^{j}, Z_{i}^{j+1})$.
\item If $u^{R}\notin Z_{i}^{j+1} $ for $\ot_{k-1}<t<\ot_{k}$, then $u^{R}=u_{\nu}(t,\og(t)+)$ is the terminal value of the $i$-jump in $\og(\ot_{k}-)$.
In particular, if $\og$ interacts at $\oP_{k}$ with another $i$-front, by the classical analysis of interactions roughly
\begin{subequations}
\bas
&s_{i}^{j}(\ot_{k}+)\geq s_{i}^{j}(\ot_{k}-)
&&\text{(interactions with shocks)}
\\
s_{i}^{j}(\ot_{k}-) -\mu^{IC}(\oP_{k})\leq&s_{i}^{j}(\ot_{k}+)\leq s_{i}^{j}(\ot_{k}-) 
&&\text{(interactions with rarefactions)}
\eas
\end{subequations}
 \end{enumerate}
 which more precisely becomes
 \ba
 \label{E:iicount}
 s_{i}^{j}(\ot_{k}-)-s_{i}^{j}(\ot_{k}+)
 \lesssim \mu^{IC}(\oP_{k}) \ .
 \ea
\end{itemize}

We can start now with the principal argument.
Collecting estimates~\eqref{E:grgrwgewgregr},~\eqref{E:egqrgregergregre},~\eqref{E:iicount} at nodes,
if $u_{\nu}^{\pm}(\og)$ between times $q$ and $t>q$ is valued strictly between $Z_{i}^{j}$ and $Z_{i}^{j+1}$ we obtain
 \begin{equation}
\label{E:rag}
 s_{i}^{j}(q)-s_{i}^{j}(t+)
 \leq  \frac{C_{2}\mu^{IC}_{\nu}\left(\left\{(r,z)\ :\ r\in(q,t]\ ,\ z=\og (r) \right\}\right)}{\beta} +C\oV (t-q+\tau_{\nu}) \ .
\end{equation}

We are now able to estimate the number of maximal, leftmost $(\beta,i,j)$-approximate sub-discontinuity curves.
Fix before intermediate times $k\ot$ where
\begin{equation}
\label{E:gargrgre}
k=0,\dots, K\ ,\quad
K:=\mathrm{ceil}\left(\frac{8\oT C\oV}{\beta}\right) \ ,
\qquad
\ot=\frac{\oT}{K}
\quad\Rightarrow\quad
\ot C\oV\leq\frac{\beta}{8}\ .
\end{equation}
Point~\ref{definitionSF3} in Definition~\ref{D:subdisc} requires that different maximal, leftmost $(\beta,i,j)$-approximate sub-discontinuity curves---when the triple is fixed---are disjoint. This disjointness yields that:
\begin{itemize}
\item At time $k\ot$ by Point~\ref{definitionSF2} in Definition~\ref{D:subdisc} there are at most $\mathrm{ceil}(4\oV/\beta)$ many of them.

\item Those  $(\beta,i,j)$-approximate sub-discontinuity curves whose interaction-cancellation measure is more than $\frac{\beta^{2}}{8C_{2}}$ are at most $\mathrm{ceil}\left(\frac{8C_{2}\mu^{IC}([0,T])\times\R)}{\beta^{2}}\right)$ by sub-additivity of measures.

\item The strength of the $j$-th component of those $(\beta,i,j)$-approximate sub-discontinuity curves $\og$ which are defined strictly between times $k\ot$ and $(k+1)\ot$ increases from a value less than $\beta/4$ at $k\ot$ to a value at least $\beta$ at some $\hat t\in (k\ot,(k+1)\ot)$, and decreases to a value less than $\beta/4$ at $(k+1)\ot$: then in some subinterval of $\left(\hat t, (k+1)\ot\right)$ estimate~\eqref{E:rag} yields
\[
\frac{3\beta}{4}< \frac{C_{2}\mu^{IC}_{\nu}(\{(r,z) \ :\  z=\og(r) \} )}{\beta} +C\oV\ot
\qquad\stackrel{\eqref{E:gargrgre}}{\Rightarrow}\qquad
\mu^{IC}_{\nu}(\{(r,z) \ :\  z=\og(r) \} )> \frac{\beta^{2}}{8C_{2}} \ .
\]
\end{itemize}
We thus estimate the number of $(\beta,i,j)$-approximate sub-discontinuity curves up to time $\oT$ by
\[
\mathrm{ceil}\left(\frac{4\oV}{\beta}\right)\cdot \mathrm{ceil}\left(\oT\frac{8C\oV}{\beta}\right)+\mathrm{ceil}\left(\frac{8C_{2}\mu^{IC}([0,T])\times\R)}{\beta^{2}}\right) 
\qedhere\]
\end{proof}
 
\subsection{Proof of the global structure of solutions}

The proof of the global structure of solutions stated in Theorem~\ref{Th:structure} above proceeds distinguishing the case of piecewise-genuinely nonlinear fields and the case of linearly degenerate fields.
The reason is not only that estimates are different in the two cases, but really the geometry of the approximation of shocks and of contact discontinuities with a family of jumps in $u_{\nu}$ is qualitatively different.
For each case, we will then have sub-cases extending the analysis in~\cite{BressanBook,BYu}.

We proceed in the next subsections with the core of the proof.

\subsubsection{Constructing the exceptional sets and limit discontinuity curves}
\label{Ss:constructionlimitcurves}
Fix a sequence $\varepsilon_{\nu}\downarrow0$ and let $u_{\nu}$ be the approximate solution of the Cauchy problem~\eqref{eq:sysnc}-\eqref{eq:inda} constructed in \S\S~\ref{sec:PCA}-\ref{S:convergenceAndExistence}.
By possibly extracting a subsequence, we can assume that the interaction and interaction-cancellation measures~\eqref{E:muIC} converge weakly* to some nonnegative measures $\mu^{I}$ and $\mu^{IC}$:
\[
\mu^{I}_{\nu}
\rightharpoonup
\mu^{I}_{}\geq0\ ,
\qquad
\mu^{IC}_{\nu}
\rightharpoonup
\mu^{IC}\geq0 \ .
\]
Of course $\mu^{IC}$ might change changing the sequence $\{u_{\nu}\}_{\nu}$.
Define now the exceptional sets
\bel{E:exceptionalSet}
\Theta_{0}:=\left\{(0,x)\quad:\quad \overline{u}(x+)\neq  \overline{u}(x-)\right\}\ ,
\qquad
\Theta_{1}:=\left\{(t,x)\quad:\quad  \mu^{IC}(\{(t,x)\})>0\right\}
\ee

\paragraph{Piecewise genuinely nonlinear fields}
Fix a threshold $\beta>0$. Suppose the $i$-th field is piecewise genuinely nonlinear. Let 
\[
\J_{\beta,i}(\nu) := \left\{ \gamma_{\nu,ik}^{} \right\}_{k=1}^{M_{\beta}}
\] 
be the family of all maximal, leftmost $(\beta,i,j)$-approximate sub-discontinuity curves in the approximate solution $u_{\nu}$ defined in \S~\ref{S:discontinuities}, if needed with repetitions of the curves. This enumeration is possible as the number these curves is uniformly bounded in $i$, $j$, $\nu$ by Lemma~\ref{L:numFronts}.

Suppose $\gamma_{\nu,ik}:(t^{-}_{\nu,ik},t_{\nu,ik}^{+})\to\R$ belong to some $\J_{\beta,i}(\nu)$ where $i,k$ are fixed, with the $i$-th family piecewise genuinely nonlinear.
One can then assume that $t_{\nu,ik}^{-}$ and $t_{\nu,ik}^{+}$ converge to $t^{-}_{ik}$ and $t^{+}_{ik}$, respectively, and that the curves $\gamma_{\nu,ik}$ converge locally uniformly on $(t^{-}_{ik} ,t ^{+}_{ik})$ to some curve $\gamma_{ik}$ by Ascoli-Arzela theorem as $\nu\uparrow\infty$.
Denote the family of such limit curves, which possibly contain repetitions, by
\[
\J_{\beta, i} := \left\{ \gamma_{ik}^{} \right\}_{k=1,\dots,M_{\beta}}  \ .
\]

\paragraph{Linearly degenerate fields}
Suppose the $i$-th field is linearly degenerate: we follow the construction in~\cite[Page~221]{BressanBook}, that we repeat for completeness.
Call $x_{i,\nu}(t;\overline t,\overline x)$ the $i$-th characteristic curve through the point $(\overline t,\overline x)$, defined by
\[
\dot x_{i,\nu}(t;\overline t,\overline x)=\lambda_{i}(u_{\nu}(t,x_{i,\nu}(t;\overline t,\overline x))) \ ,
\quad  x_{i,\nu}(\overline t;\overline t,\overline x)=\overline x \ .
\]
This is allowed by linear degeneracy of the $i$-th field.
Up to extracting a subsequence, one can assume that 
\[
x_{i,\nu}(t;\overline t,\overline x) \quad\to\quad x_{i}(t;\overline t,\overline x)
\qquad\text{as $\nu\uparrow\infty$.}
\]
Denote by $\overline\mu_{\nu}^{i\pm}$ the measures of the positive and negative $i$-waves in the approximation ${u}_{\nu}(0+,\cdot)$ of the initial datum $\overline{u}$ and by $\overline{\mu}^{i\pm}$ the measures of the positive and negative $i$-waves in the initial datum $\overline{u}$.
As $\nu\uparrow\infty$, by suitably choosing $\overline u_{\nu}$ we can assume the weak*-convergence
\[
\overline\mu_{\nu}^{i\pm}
\rightharpoonup
\overline{\mu}_{}^{i\pm} \ .
\]
Define the measure of the total amount of positive and negative $i$-waves on $u(0,\cdot)$ present in an interval $(a,b)$ plus the total amount of interaction and cancellation occurring in the corresponding forward strip as
\[
\mu^{i*}_{\nu}((a,b)):=\overline{\mu}_{}^{i+}((a,b))+\overline{\mu}_{}^{i-}((a,b))+\mu^{IC}\left(\{(t,x) \ :\ t>0,\ x_{i,\nu}(t;0,a)<x<x_{i,\nu}(t;0,b)\}\right) \ .
\]
We finally define for positive times the family of \emph{at most} countably many curves
\begin{subequations}
\label{E:contactdisccurves}
\begin{equation}
\J_{i}=\{\gamma_{i m}:=x_{i}(\cdot ;0,\xi_{im}) \quad:\quad \mu^{i*}(\{\xi_{im}\})>0 \ ,\quad m\in\nat\} 
\end{equation}
where $ \mu^{i*}$ are constructed as follows.
By upper semicontinuity properties of $w^{*}$-convergence of nonnegative measures, we define
points
\begin{equation}\label{E:dfggfddffe}
\xi_{\nu,im}^{\pm}\to\xi_{\nu,im}
\quad\text{such that}\quad
\mu^{i*}_{\nu}\left([\xi_{\nu,im}^{-},\xi_{\nu,im}^{+}]\right) \xrightarrow{\nu\uparrow\infty} \mu^{i*}_{}\left(\{\xi_{im}\}\right) 
\quad\text{as $\nu\uparrow\infty$}\quad
\end{equation}
and consequently approximating regions $\{\gamma_{\nu,i m}^{-}(t)\leq x\leq \gamma_{\nu,i m}^{+}(t)\}$ where the curves $\gamma_{\nu,i m}^{\pm}$ are
\begin{equation}
\J_{i}^{-}(\nu):=\left\{\gamma_{\nu,i m}^{-}:=x_{i,\nu}\left(\cdot ;0,\xi_{\nu,im}^{-}\right) \right\}\ ,
\quad
\J_{i}^{+}(\nu):=\left\{\gamma_{\nu,i m}^{-}:=x_{i,\nu}\left(\cdot ;0,\xi_{\nu,im}^{+}\right) \right\}
\ .
\end{equation}
\end{subequations}
Notice that we can also assume that $\gamma_{\nu,i m}^{-}$ and $\gamma_{\nu,i m}^{+}$ converge to $\gamma_{i m}$ locally uniformly.
We stress that, differently from the case of shocks for piecewise genuinely nonlinear fields, contact discontinuities in $u$ are not necessarily approximated by discontinuities in $u_{\nu}$ with strength definitively above a fixed threshold: they might be approximated by an increasing number of small discontinuities of $u_{\nu}$ in a region shrinking to the limit curve.

\paragraph{Exceptional set and family of limit curves}
Denote by $\Theta_{2}$ the subset of the plane where two limit curves $ \gamma_{ik}$, $ \gamma_{i'k'}$ belonging to different characteristic families $i\neq i'$ cross each-other:
\[
\Theta_{2} = \left\{\left(\ot,\ox\right)\ :\ \exists i,i',k,k'\quad i\neq i'\ |\ \ox= \gamma_{ik}\left(\ot\right)= \gamma_{i'k'}\left(\ot\right)\right\}
\]
We list $\Theta_{2}$ separately for being more explicit, but by Lemma~\ref{L:distinctf} below one proves $\Theta_{2}\subset\Theta_{1}$.
Define the exceptional set $\Theta$ and the family of curves $\J$ in the statement of Theorem~\ref{Th:structure} as
\[
\Theta:=\Theta_{0}\cup \Theta_{1} \cup \Theta_{2} \ ,
\qquad
\J:=\bigcup_{\substack{i\text{ linearly}\\\text{degenerate}}}\J_{i}\cup  \bigcup_{\substack{i\text{ pw genuinely}\\\text{nonlinear}}}\cup_{n\in\nat}\J_{\frac{1}{n},i} 
\ ,
\]
where $\Theta_{0}$, $ \Theta_{1} $, $ \Theta_{2}$ and $\J_{i}$, $\J_{\beta,i}$ are defined just above.

\subsubsection{Jumps: Piecewise-genuinely nonlinear fields}
\label{Ss:pwgnflimit}

Let $\oP=(\ot,\ox)\notin\Theta$ be a point along a curve $\og\in \J_{\beta}$ of a piecewise-genuinely nonlinear family $i$: since $u_{\nu}$ is converging to $u$ in $L^{1}$ and by construction of these limit curves, in $\oP$ there must be a jump of the limit function $u$ of strength at least $\beta/4$---see Definition~\ref{D:subdisc}.
We prove below the sided limits of $u$ at $\oP$ stated in Theorem~\ref{Th:structure}, while instead the slope and entropy condition of the jump can be then deduced precisely following~\cite[Step~7, Page 227]{BressanBook}, therefore we omit the proof here.

\paragraph{Simple jumps} Let $\oP=(\ot,\ox)$ be a point along a simple $i$-shock curve $\og\in \J_{\beta}$ of a piecewise-genuinely nonlinear family such that $\mu^{IC}(\{\oP\})=0$.
Even if we do not need it, we remind that by the tame oscillation condition---see~\cite[Lemma~2.3]{Ch2} jointly with Theorem~\ref{T:localEst}---one can define
\[
u^{L} = \lim_{\substack{ (t,x)\to(\ot,\ox)\\ \ot\leq t \leq \ot+(\ox-x)/\hat\lambda}} u(t,x) \ .
\]
Let $\og_{\nu}$ be a maximal, leftmost $(\beta,i,j)$-approximate sub-discontinuity curve converging to $\og$.

Suppose that Theorem~\ref{Th:structure} fails: equivalently, suppose by contradiction that
\[
\lim_{r\downarrow0}\limsup_{\nu\uparrow\infty} \left( \sup_{\substack{ |t-\ot|+|x-\ox|\leq r \\ x<\og_{\nu}}} \left| u_{\nu}(t,x)-u^{L} \right|\right) 
>
\varepsilon
> 0 \ .
\]
By possibly extracting a subsequence and supposing that we suitably normalize the pointwise representative of the solution, we directly assume there are points $P_{\nu}=(\ot_{\nu},\op_{\nu})$ and $Q_{\nu}=(\ot_{\nu},\oq_{\nu})$ on the left of $\oP$, i.e.~satisfying $\op_{\nu}<\oq_{\nu}<\og_{\nu}(\ot_{\nu})$, such that
\[
Q_{\nu}\xrightarrow{\nu\uparrow\infty} \oP\ ,
\qquad
|u_{\nu}(Q_{\nu})-u^{L}|\geq\varepsilon \quad\forall\nu,
\]\[
P_{\nu}\xrightarrow{\nu\uparrow\infty} \oP\ ,
\qquad
|u_{\nu}(P_{\nu})-u^{L}|\xrightarrow{\nu\uparrow\infty}0\ . \quad
\]
The segment $P_{\nu}Q_{\nu}$ must then be crossed by a large amount of waves.
As in~\cite{BressanBook,BYu} we show below that in any neighborhood of $\oP$ these waves either interact among themselves or with $\og_{\nu}$. As a result, $\mu^{IC}(\{\oP\})$ cannot vanish and therefore we reach an absurd.
Two cases are possible:

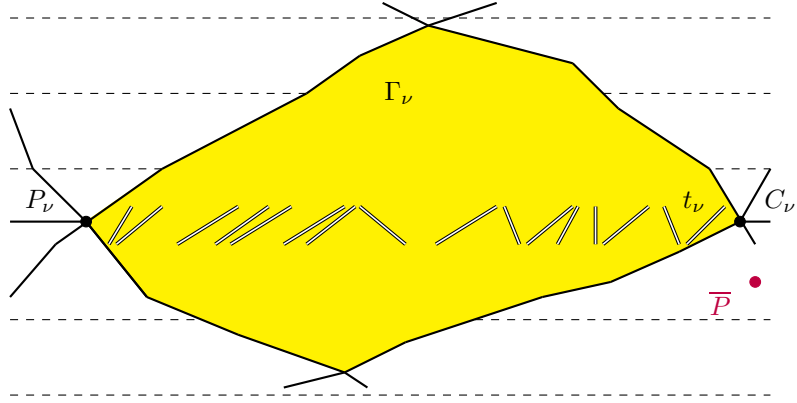
\begin{figure}[ht!]\centering
  \begin{tikzpicture}
\draw[ thick] (0,.3) -- (10,.3);

\draw (0.4,.3) node[ above ,font=\normalsize] {$P_{\nu}$};
\draw (9.8,.3)  node[ above right ,font=\normalsize] {$C_{\nu}$};
\draw (9.6,-.5) node[ below left ,font=\normalsize,purple] {$\oP$};
\filldraw[purple] (9.8,-.5) circle(2pt);

\draw[ dashed] (0,1) -- (10,1);
\draw[ dashed] (0,2) -- (10,2);
\draw[ dashed] (0,3) -- (10,3);
\draw[ dashed] (0,-1) -- (10,-1);
\draw[ dashed] (0,-2) -- (10,-2);

\filldraw[yellow] (
(1,.3) -- (1.8,-.7)  -- (3,-1.2) -- (4.4,-1.7) -- (5.2,-1.3) -- (7,-.7) -- (7.90,-.5) -- (8.8,-.1) -- ( 9.6,.3) %
-- (9.2,1) -- (8,1.8) -- (7.4,2.4) -- (5.5,2.9) -- (4.6,2.5) -- (3.9,2) -- (2,1) -- cycle;

\draw[ thick]  (0,-.7) -- (0.6,0);
\draw[ thick] (0.6,0) -- (2,1);
\draw[ thick] (2,1) -- (3.9,2);
\draw[ thick] (3.9,2) -- (4.6,2.5);
\draw[ thick] (4.6,2.5) -- (5.5,2.9);
\draw[ thick] (4.9,3.2) -- (5.5,2.9);
\draw[ thick] (6.4,3.2) -- (5.5,2.9);
\draw[ thick] (5.5,2.9) -- (7.4,2.4);
\draw[ thick] (7.4,2.4) -- (8,1.8);
\draw[ thick] (8,1.8) -- (9.2,1);
\draw[ thick] (9.2,1) -- (9.8,0);
%
\draw[ thick] (0,1.8) -- (.3,1);
\draw[ thick] (.3,1) -- (1,.3);
\draw[ thick] (1,.3) -- (1.8,-.7);
\draw[ thick] (1,.3) -- (1.8,-.7);
\draw[ thick] (1.8,-.7) -- (3,-1.2);
\draw[ thick] (3,-1.2) -- (4.4,-1.7);
\draw[ thick] (3.6,-1.9) -- (4.4,-1.7); 
\draw[ thick] (4.7,-1.9) -- (4.4,-1.7); 
\draw[ thick] (4.4,-1.7) -- (5.2,-1.3);
\draw[ thick] (5.2,-1.3) -- (7,-.7);
\draw[ thick] (7,-.7) -- (7.90,-.5);
\draw[ thick] (7.90,-.5) -- (8.8,-.1);
\draw[ thick] (8.8,-.1) -- ( 9.6,.3);
\draw[ thick] ( 9.6,.3) -- (10,1);

\draw[ double ] (1.3,0) -- (1.6,.5);
\draw[  double] (1.4,0) -- (2,.5);
\draw[double  ] (2.2,0) -- (3,.5);
\draw[double  ] (2.7,0) -- (3.4,.5);
\draw[double  ] (2.9,0) -- (3.7,.5);
\draw[double  ] (3.6,0) -- (4.4,.5);
\draw[double  ] (3.9,0) -- (4.54,.5);
\draw[  double] (5.2,0) -- (4.6,.5);
\draw[double  ] (5.6,0) -- (6.4,.5);
\draw[double  ] (6.7,0) -- (6.5,.5);
\draw[double  ] (6.8,0) -- (7.4,.5);
\draw[double  ] (7.2,0) -- (7.47,.5);
\draw[double  ] (7.7,0) -- (7.7,.5);
\draw[double  ] (7.8,0) -- (8.4,.5);
\draw[double  ] (8.9,0) -- (9.4,.5);
\draw[double  ] (8.8,0) -- (8.6,.5);


\draw (4.8,2) node[right,font=\normalsize] {$\Gamma_{\nu}$};
\filldraw (9.6,.3) circle(2pt);
\filldraw (1,.3) circle(2pt);
\draw (9,.3) node[above,font=\normalsize] {$t_{\nu}$};
\end{tikzpicture}

\caption{Lemma~\ref{L:distinctf}: there is a large amount of waves in at least two distinct families.}
\label{fig:regioneDist}
\end{figure}
\paragraph{Case 1 (Figure~\ref{fig:regioneDist})}
Each segment $P_{\nu}Q_{\nu}$ is crossed by a fixed amount of $k$-waves in $u_{\nu}$ for some fixed $k\neq i$. 
Replacing $Q_{\nu}$ by the point $\left(\ot_{\nu}, \og_{\nu}(\ot_{\nu})+c_{\nu}\right)$, for a suitable $c_{\nu}\downarrow0$, Lemma~\ref{L:distinctf} below applies with the families $i$ and $k$ and it contradicts the assumption $\mu^{IC}(\{\oP\})=0$, i.e.~$\oP\notin\Theta_{2}$, because there is a positive amount of interactions among different families.

\paragraph{Case 2 (Figure~\ref{fig:regione1})}
Suppose that
\begin{itemize}
\item the $j$-th component of $i$-fronts in $u_{\nu}$ other than $\og_{\nu}$, crossing each segment $P_{\nu}Q_{\nu}$ has total strength more than $\varepsilon$ but 
\item in any region $\overline\Gamma_{\nu}$ shrinking to $\oP$ when ${\nu}\uparrow\infty$ for all $k\neq i$ the amount of $k$-waves converges to $0$ and 
\item in any region $\overline\Gamma_{\nu}$ shrinking to $\oP$ also the maximum strength of the $j$-th component of $i$-fronts other than $\og_{\nu}$ is vanishingly small.
\end{itemize}
Lemma~\ref{L:samef} below then applies yielding that there is a uniformly positive amount of interaction-cancellation in the $i$-th family in $\oP$: it contradicts the assumption $\oP\notin\Theta_{2}$, since this would imply $\mu^{IC}(\{\oP\})=0$, and therefore it ends the proof.

%
%
\begin{figure}\centering
\hfill
  \begin{tikzpicture}
\draw[ thick] (0,.3) -- (5,.3);

\draw (4.5,.3) node[above,font=\normalsize] {$t_{\nu}$};
\draw (4.6,4.7) node[left,font=\normalsize] {$\ot_{\nu}$};
\draw (0.3,.3) node[ above ,font=\normalsize] {$P_{\nu}$};
\draw (2.4,.3)  node[ above right ,font=\normalsize] {$C_{\nu}$};
\draw(1.3,1.5)  node[ left  ,font=\normalsize] {$\gamma^{\ell}_{\nu}$};
\draw(2.9,1.5)  node[ right  ,font=\normalsize] {$\gamma^{r}_{\nu}$};
\draw (2.7,-.5) node[ above ,font=\normalsize,purple] {$\oP$};
\filldraw[purple] (2.3,-.5) circle(2pt);

\draw[ dashed] (0,1) -- (5,1);
\draw[ dashed] (0,2) -- (5,2);
\draw[ dashed] (0,3) -- (5,3);
\draw[ dashed] (0,4) -- (5,4);
\draw[ dashed] (0,5) -- (5,5);

\filldraw[yellow] (0.3,0) -- (1,1)  -- (1.95,2) -- (2.3,2.5) -- (3.5,3.7) -- (4,4) -- (4.9,4.8) %
			-- (4.4,4)  -- (3.95,3.5) -- (3.6,3) -- (3.2,2) -- (2.5,1) -- (2.1,0) -- cycle;

\draw[ thick] (0.3,0) -- (1,1);
\draw[ thick] (1,1) -- (1.95,2);
\draw[ thick] (1.95,2) -- (2.3,2.5);
\draw[ thick] (2.3,2.5) -- (2.75,3);
\draw[ thick] (2.75,3) -- (3.5,3.7);
\draw[ thick] (3.5,3.7) -- (4,4);
\draw[ thick] (4,4) -- ( 4.9,4.8);

\draw[ thick] (2.1,0) -- (2.51,1);
\draw[ thick] (2.5,1) -- (3.2,2);
\draw[ thick] (3.2,2) -- (3.6,3);
\draw[ thick] (3.6,3) -- (3.95,3.5);
\draw[ thick] (3.95,3.5) -- (4.4,4);
\draw[ thick] (4.4,4) -- ( 4.9,4.8);

\draw[ double ] (.4,0) -- (.8,.5);
\draw[  double] (.7,0) -- (1,.5);
\draw[  double] (.9,0) -- (1.2,.5);
\draw[double  ] (1.1,0) -- (1.5,.5);
\draw[double  ] (1.4,0) -- (1.7,.5);
\draw[double  ] (1.55,0) -- (1.85,.5);
\draw[double  ] (1.8,0) -- ( 2.2,.5);

\draw[ double ] (4.1,3.3) -- (3.95,3.5);
\draw[ double ] (3.95,3.5) -- (3.5,3.7);
\draw[ double ] (3.5,3.7) -- (3,4);

\draw[ double ] (1.9,2.2) -- (2.3,2.5);

\draw (2.4,2) node[right,font=\normalsize] {$\Gamma_{\nu}$};
\filldraw (2.21,.3) circle(2pt);
\filldraw (.51,.3) circle(2pt);
\end{tikzpicture}
%
\hfill
 \begin{tikzpicture}
\draw[ thick] (0,-.3) -- (-5,-.3);

\draw (-4.5,-.3) node[below,font=\normalsize] {$t_{\nu}$};
\draw (-4.6,-4.7) node[right,font=\normalsize] {$\ot_{\nu}$};
\draw (-0.3,-.3) node[ below ,font=\normalsize] {$C_{\nu}$};
\draw (-2.4,-.3)  node[ below left ,font=\normalsize] {$P_{\nu}$};
\draw(-1.4,-1.5)  node[ right  ,font=\normalsize] {$\gamma^{r}_{\nu}$};
\draw(-2.9,-1.5)  node[ left  ,font=\normalsize] {$\gamma^{\ell}_{\nu}$};
\draw (-.3,.3) node[  ,font=\normalsize,purple] {$\oP$};
\filldraw[purple] (-.05,.5) circle(2pt);

\draw[ dashed] (0,-1) -- (-5,-1);
\draw[ dashed] (0,-2) -- (-5,-2);
\draw[ dashed] (0,-3) -- (-5,-3);
\draw[ dashed] (0,-4) -- (-5,-4);
\draw[ dashed] (0,-5) -- (-5,-5);

\filldraw[yellow] (-0.3,0) -- (-1,-1)  -- (-1.95,-2) -- (-2.3,-2.5) -- (-3.5,-3.7) -- (-4,-4) -- (-4.9,-4.8) %
			-- (-4.4,-4)  -- (-3.95,-3.5) -- (-3.6,-3) -- (-3.2,-2) -- (-2.5,-1) -- (-2.1,0) -- cycle;

\draw[ thick] (-0.3,0) -- (-1,-1);
\draw[ thick] (-1,-1) -- (-1.95,-2);
\draw[ thick] (-1.95,-2) -- (-2.3,-2.5);
\draw[ thick] (-2.3,-2.5) -- (-2.75,-3);
\draw[ thick] (-2.75,-3) -- (-3.5,-3.7);
\draw[ thick] (-3.5,-3.7) -- (-4,-4);
\draw[ thick] (-4,-4) -- (-4.9,-4.8);

\draw[ thick] (-2.1,0) -- (-2.51,-1);
\draw[ thick] (-2.5,-1) -- (-3.2,-2);
\draw[ thick] (-3.2,-2) -- (-3.6,-3);
\draw[ thick] (-3.6,-3) -- (-3.95,-3.5);
\draw[ thick] (-3.95,-3.5) -- (-4.4,-4);
\draw[ thick] (-4.4,-4) -- (-4.9,-4.8);

\draw[ double ] (-.4,0) -- (-.8,-.5);
\draw[  double] (-.7,0) -- (-1,-.5);
\draw[  double] (-.9,0) -- (-1.2,-.5);
\draw[double  ] (-1.1,0) -- (-1.5,-.5);
\draw[double  ] (-1.4,0) -- (-1.7,-.5);
\draw[double  ] (-1.55,0) -- (-1.85,-.5);
\draw[double  ] (-1.8,0) -- (-2.2,-.5);

\draw[ double ] (-4.1,-3.3) -- (-3.95,-3.5);
\draw[ double ] (-3.95,-3.5) -- (-3.5,-3.7);
\draw[ double ] (-3.5,-3.7) -- (-3,-4);

\draw[ double ] (-2.5,-3.1) -- (-2.3,-2.5);

\draw (-2.8,-2) node[right,font=\normalsize] {$\Gamma_{\nu}$};
\filldraw (-2.21,-.3) circle(2pt);
\filldraw (-.51,-.3) circle(2pt);

\end{tikzpicture}
\hfill\phantom{,}

\caption{Lemma~\ref{L:samef}. There might be one $(\beta,i,j)$-approximate sub-discontinuity curve merging.}
\label{fig:regione1}
\end{figure}
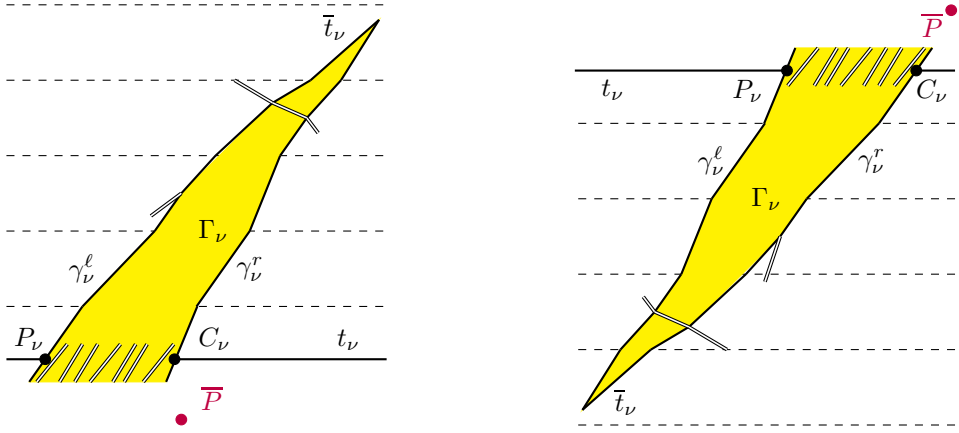

\paragraph{Composite waves} 
If the limit wave $\og$ is composite, then one can apply the argument above to each sub-component of the jump, as in~\cite{BYu}, since Lemma~\ref{L:distinctf} and Lemma~\ref{L:samef} below still apply.
Of course, the limiting value of $u$ depends on the $j$-th component we are considering and therefore on the respective $(t,x)$-region where the limit is taken.

\subsubsection{Jump points: Linearly degenerate fields}
\label{Ss:rfroionrgnjg}

Suppose now the $i$-th characteristic field is linearly degenerate.
Let $\oP=(\ot,\ox)$ be a point along a curve $\og\in \J$ as constructed in~\eqref{E:contactdisccurves} and such that $\mu^{IC}(\{\oP\})=0$. In particular, thus, $\oP\notin\Theta$.
Let moreover $\og_{\nu}^{-}\in \J^{-}(\nu)$ and $\og_{\nu}^{+}\in \J^{+}(\nu)$ be corresponding curves delimitating the approximating region which converge locally uniformly to $\bar \gamma$, again as constructed in~\eqref{E:contactdisccurves}.
Even if we do not need it, we remind that by the tame oscillation condition---see~\cite[Lemma~2.3]{Ch2} jointly with Theorem~\ref{T:localEst}---one can define
\[
u^{L} = \lim_{\substack{ (t,x)\to(\ot,\ox)\\ \ot\leq t \leq \ot+(\ox-x)/\hat\lambda}} u(t,x) \ .
\]
We prove below the sided limits of $u$ at $\oP$ stated in Theorem~\ref{Th:structure}.
We omit here, instead, the proof of the relation $\dot {\overline{y}}(t)=\sigma_{i}(u^{L},u^{R})$ because it can be deduced as in~\cite[Step~7, Page 227]{BressanBook}.

Suppose that Theorem~\ref{Th:structure} fails: equivalently, suppose by contradiction that
\[
\lim_{r\downarrow0}\limsup_{\nu\uparrow\infty} \left( \sup_{\substack{ |t-\ot|+|x-\ox|\leq r \\ x<\og_{\nu}^{-}}} \left| u_{\nu}(t,x)-u^{L} \right|\right) > \varepsilon > 0 \ .
\]
By possibly taking a subsequence and supposing that we suitably normalize the pointwise representative of the solution, we assume there are points $Q_{\nu}=(\ot_{\nu},\oq_{\nu})$ and $P_{\nu}=(\ot_{\nu},\op_{\nu})$ on the left of $\og_{\nu}^{-}$, i.e.~satisfying $\op_{\nu}<\oq_{\nu}<\og_{\nu}^{-}(\ot_{\nu})$, such that $\ot_{\nu}\to\ot$ when $\nu\uparrow\infty$ and
\[
Q_{\nu}\xrightarrow{\nu\uparrow\infty} \oP\ ,
\qquad
|u_{\nu}(Q_{\nu})-u^{L}|\geq\varepsilon\quad\forall\nu,
\]\[
P_{\nu}\xrightarrow{\nu\uparrow\infty} \oP\ ,
\qquad
|u_{\nu}(P_{\nu})-u^{L}|\xrightarrow{\nu\uparrow\infty}0\ . \quad
\]
By the construction of the limit curve for linearly degenerate families of~\S~\ref{Ss:constructionlimitcurves}:
\begin{itemize}
\item For every linearly degenerate field $j\neq i$ the total strength of $j$-waves crossing $P_{\nu}Q_{\nu}$ converges to $0$ as $\nu\uparrow\infty$: see~\cite[(10.76)]{BressanBook} for a full proof.
\item The amount of waves of the $i$-th family crossing $P_{\nu}Q_{\nu}$ is vanishingly small, as a consequence of~\eqref{E:dfggfddffe}, because we fixed $P_{\nu}$ and $Q_{\nu}$ converging to $\oP$ but on the left of $\og_{\nu}^{-}$.
\end{itemize}
Nevertheless, the segment $P_{\nu}Q_{\nu}$ must be crossed by a large amount of waves, which must therefore be of piecewise genuinely nonlinear families.
We show that in any neighborhood of $\oP$ these waves must interact among themselves. 
As a result, $\mu^{IC}(\{\oP\})$ cannot vanish and therefore we reach an absurd.
Two cases are possible:

\paragraph{Case 1 (Figure~\ref{fig:regioneDist})}
There are two distinct indexes $j\neq k$ such that each segment $P_{\nu}Q_{\nu}$ is crossed by a fixed amount of both $j$-waves and $k$-waves in $u_{\nu}$. Lemma~\ref{L:distinctf} below applies with the families $k$ and $j$ and it contradicts the assumption $\mu^{IC}(\{\oP\})=0$ because of interactions among different families.

\paragraph{Case 2 (Figure~\ref{fig:regione1})}
For a singe index $j\in\{1,\dots,n\}$ each segment $P_{\nu}Q_{\nu}$ is crossed by an amount $\geq \varepsilon>0$ of waves of a piecewise genuinely nonlinear family $j$, but for $k\neq j$ the total amount of $k$-waves crossing the segment $P_{\nu}Q_{\nu}$ vanishes as $\nu\to \infty$.
As in the limit there is no $(\beta,j,\ell)$-shock front at $\oP$, then one must have that the maximum strength of $j$-waves crossing the segment $P_{\nu}Q_{\nu}$ vanishes as $\nu\to \infty$ and therefore that the segment $P_{\nu}Q_{\nu}$ is crossed by a large number of small $j$-waves.
Lemma~\ref{L:samef} below applies and it contradicts the assumption $\mu^{IC}(\{\oP\})=0$.

\subsubsection{Continuity points}
\label{Ss:continuityPoints}

Consider now a point $\oP\notin\Theta$ which does not belong to any of the curves in $\J$ constructed in \S~\ref{Ss:constructionlimitcurves}. We prove now that $u$ is continuous at $\oP$, concluding thus the statement of Theorem~\ref{Th:structure}.

Assume by contradiction that $u$ is discontinuous at $\oP$: there exists $\varepsilon>0$ and a space-like segment $P_{\nu} Q_{\nu}$ degenerating to the singe point $\oP$ for which 
\[
u_{\nu}(P_{\nu})\to u(\oP) \ ,
\qquad
|u_{\nu}(Q_{\nu})-u(\oP)|\geq \varepsilon\quad\forall\nu\ .
\]
Two cases are possible.

\paragraph{Case 1 (Figure~\ref{fig:regioneDist})}
There are two distinct indexes $j\neq k$ such that each segment $P_{\nu}Q_{\nu}$ is crossed by a fixed amount of both $j$-waves and $k$-waves in $u_{\nu}$. Lemma~\ref{L:distinctf} below applies with the families $k$ and $j$ and it contradicts the assumption $\mu^{IC}(\{\oP\})=0$ because of interactions among different families.

\paragraph{Case 2 (Figure~\ref{fig:regione1})}
For a singe index $j\in\{1,\dots,n\}$ each segment $P_{\nu}Q_{\nu}$ is crossed by an amount $\geq \varepsilon>0$ of $j$-waves, but for $k\neq j$ the total amount of $k$-waves crossing the segment $P_{\nu}Q_{\nu}$ vanishes.

Suppose the family $j$ is piecewise genuinely nonlinear.
As in the limit there is no $(\beta,j,\ell)$-shock front, then one must have that the maximum strength of $j$-waves crossing the segment $P_{\nu}Q_{\nu}$ vanishes as $\nu\to \infty$ and therefore that the segment $P_{\nu}Q_{\nu}$ is crossed by a large number of small $j$-waves.
Lemma~\ref{L:samef} below applies and it contradicts the assumption $\mu^{IC}(\{\oP\})=0$. 

It is moreover not possible that the family $j$ is linearly degenerate unless $\mu^{j*}_{\nu}(\oP)>0$ see~\cite[(10.76)]{BressanBook} for a full proof.
This contradicts the assumption that $\oP$ is not covered by any curve in $\J$.

\subsubsection{Geometric lemmas on piecewise genuinely nonlinear characteristic fields}
\label{Ss:geomLemmas}

 We now formalize the following intuitive but nontrivial fact. The strength of the $i$-th wave in the solution of the Riemann problem among $v(0)$ (left value) and $v(1)$, for a piecewise constant function $v:[0,1]\to \R^{N}$ whose jumps are mostly in the $i$-th family and small except possibly for one of them, is approximatively the sum of the strengths of the $i$-th waves in its jumps. For example, jumping many times almost along the $i$-th elementary curve, and staying ``close'', if measured along the $i$-th elementary curve, to the initial point, one does not move transversally too much.  

\begin{lemma}
\label{L:acc2}
There exists $C>0$ which satisfies the following.
Let $\rho>0$ and let $v:[0,1]\to \R^{N}$ be a piecewise constant function such that \[Ce^{C\TV(v)}\TV(v)<\dH\] and the following conditions hold:
\begin{itemize}
\item the total strength of waves of families different from $i$, present at any jump, is less than $\rho$;
\item the strength of the single $i$-th wave in the solution of the Riemann problem at each jump, apart from at most one of them, is less than $\rho$.
\end{itemize}
Denote by
\begin{itemize}
\item $s_{i}^{k}$ the strength of the $i$-th wave present in the solution of the Riemann problem relative to the $k$-th jump of $v$;
\item $\sigma_{1},\dots,\sigma_{N}$ the strengths of the outgoing waves relative to the jump among $v(0)$ (left value) and $v(1)$ (right value).
\end{itemize}
We conclude then that
\[
\left|\sigma_{i}-\sum_{k}s_{i}^{k}\right|+\sum_{j\neq i}|\sigma_{j}|\leq C e^{\displaystyle {C\,\TV v}}\cdot (1+\TV v)\cdot\rho \ .
\]
\end{lemma}
\begin{proof}
The thesis holds if we prove it when all jumps of $v$ belong to the $i$-th characteristic family, by classical interaction estimates, since the total strength of waves of all families different from $i$ is assumed to be less than $\rho$.

Suppose therefore that all jumps of $v$ are along the $i$-th elementary waves.
In this case, $s_{i}^{k}$ already denotes the strength of the $k$-th jump. We argue by induction on the number $K$ of jumps of $v$ that, in the absence of waves of other families,
\begin{equation}
\label{E:sarqg}
\left|\sigma_{i}-\sum_{k}s_{i}^{k}\right|+\sum_{j\neq i}|\sigma_{j}|
\leq 
Ce^{\displaystyle {C\,(|s_{i}^{1}|+\dots+|s_{i}^{K}|)}} \cdot(|s_{i}^{1}|+\dots+|s_{i}^{K}|)\cdot\rho \ .
\end{equation}
Of course when $K=1$ the thesis is trivial while if $K=2$ by classical interaction estimates~\cite[Lemma~1]{AMfr}, and since either $|s_{i}^1|\leq\rho$ or $|s_{i}^2|\leq\rho$ by assumption, one has that the thesis holds with $C>C_{2}$ where
\begin{equation*}
\left|\sigma_{i}-(s_{i}^{1}+s_{i}^{2})\right|+\sum_{j\neq i}|\sigma_{j}|
\leq 
C_{1}\I(s_{i}^{1},s_{i}^{2})\leq C_{1}|s_{i}^{1}s_{i}^{2}| 
\leq 
C_{2}e^{\displaystyle {C_{2}\,(|s_{i}^{1}|+ |s_{i}^{2}|)}} \cdot (|s_{i}^{1}|+|s_{i}^{2}|)\cdot\rho \ .
\end{equation*}
Suppose now that estimate~\eqref{E:sarqg} holds if $v$ has $K$ jumps. Denote by $\sigma^{*}_{1},\dots,\sigma^{*}_{N}$ the strengths of the outgoing waves relative to the Riemann problem among $v(0)$ (left value) and the second-to-last value $v^{*}$ of $v$ (right value).
By the induction hypothesis we have the estimate
\be
\label{E:grerga}
\left|\sigma^{*}_{i}-\sum_{k=1}^{K-1}s_{i}^{k}\right|+\sum_{j\neq i}|\sigma^{*}_{j}|
\leq 
Ce^{\displaystyle {C\,(|s_{i}^{1}|+\dots+|s_{i}^{K-1}|)}} \cdot (|s_{i}^{1}|+\dots+|s_{i}^{K-1}|)\cdot\rho \ .
\ee
Moreover, by the estimates on interactions among two consecutive Riemann problems \cite[Theorem 3.7]{sie} applied to the states $v(0)$ (left), $v^{*}$ (middle) and $v(1)$ (right) one has that the strengths $\sigma_{1},\dots,\sigma_{N}$ in the outgoing waves of the Riemann problem among $v(0)$ and $v(1)$ satisfy
\[
 \left|\sigma^{}_{i}-(\sigma^{*}_{i}+s_{i}^{K})\right| +\sum_{j\neq i}\left|\sigma^{}_{j}-\sigma^{*}_{j}\right| 
 \leq
\overline C |s_{i}^{K}| \sum_{j\geq i} |\sigma^{*}_{j}|  \ .
\]
By the last estimate and by the triangular inequalities
\bas
&\left|\sigma^{}_{i}-\sum_{k=1}^{K}s_{i}^{k}\right| \leq |\sigma^{}_{i} - (\sigma^{*}_{i}+s^{K}_{i}) |
+\left|\sigma^{*}_{i}+s^{K}_{i}-\sum_{k=1}^{K}s_{i}^{k}\right|
&&
\sum_{j\neq i}\left|\sigma^{}_{j} \right| =\sum_{j\neq i}\left|\sigma^{}_{j}- \sigma^{*}_{j}\right| +\sum_{j\neq i}\left|\sigma^{*}_{j} \right| 
\eas
we get
\begin{align*}
 \left|\sigma^{}_{i}-\sum_{k}s_{i}^{k}\right| +\sum_{j\neq i}\left|\sigma^{}_{j} \right| 
& \leq
 \left|\sigma^{}_{i}-(\sigma^{*}_{i}
 +s_{i}^{K})\right|+\left|\sigma^{*}_{i}-\sum_{k=1}^{K-1}s_{i}^{k}\right|
 +\sum_{j\neq i}\left|\sigma^{}_{j}-\sigma^{*}_{j} \right| +\sum_{j\neq i}|\sigma^{*}_{j}|
\\
& \leq
\left|\sigma^{*}_{i}-\sum_{k=1}^{K-1}s_{i}^{k}\right|+\sum_{j\neq i}|\sigma^{*}_{j}|+\overline C |s_{i}^{K}| \sum_{j> i} |\sigma^{*}_{j}| 
  \ .
\end{align*}
Since $C e^{C\TV(v)}\TV(v)\leq 1$ implies jointly with \eqref{E:grerga} that \[\sum_{j> i} |\sigma^{*}_{j}| \leq\sum_{j\neq i} |\sigma^{*}_{j}| \leq  \rho\ ,\]  then the induction hypothesis~\eqref{E:grerga} yields 
\begin{align*}
 \left|\sigma^{}_{i}-\sum_{k}s_{i}^{k}\right| +\sum_{j\neq i}\left|\sigma^{}_{j} \right| 
& \leq
 \overline C |s_{i}^{K}|\rho  +Ce^{\displaystyle {C\,(|s_{i}^{1}|+\dots+|s_{i}^{K-1}|)}} \cdot (|s_{i}^{1}|+\dots+|s_{i}^{K-1}|)\cdot\rho 
 \\
 & \leq Ce^{\displaystyle {C\,(|s_{i}^{1}|+\dots+|s_{i}^{K}|)}} \cdot (|s_{i}^{1}|+\dots+|s_{i}^{K}|)\cdot\rho 
 \ .\qquad\qquad\qedhere
\end{align*}
\end{proof}

We state an elementary lemma on the geometric structure of piecewise-genuinely nonlinear characteristic fields.
We recall that by piecewise-genuine-nonlinearity $\nabla\lambda_{i}(u)\cdot r_{i}(u)$ vanishes only on $Z_{i}^{k}$ for $k=1,\dots,J_{i}$, which are hypersurfaces transversal to the $i$-rarefaction curves. 

\begin{lemma}
\label{L:acc0}
Let $\rho, M>0$. Suppose the $i$-th field is piecewise genuinely nonlinear.
Then there exists a positive constant $\overline C$ such that the following holds.
Suppose $v$ is a piecewise constant function as in Lemma \ref{L:acc2} and whose image lies in a connected component of the compact set
\[
D:=[-M,M]^{N}\setminus \bigcup_{k=1}^{J_{i}} B_{\rho}(Z_{k}^{i}) \ .
\]
Define $t_{1},t_{2}$ and $\ell,L>0$ by the relations
\[
L=\mathrm{diam}\,\mathrm{Image}\, v_{i} =|v_{i}(t_{1})-v_{i}(t_{2})| \ ,
\qquad
\ell= \min \left\{ \left|\nabla\lambda_{i}(u)\cdot r_{i}(u)\right| \quad:\quad u\in D\right\}\ .
\]
One has then that \quad $ | \lambda_{i}(v(t_{1}))-\lambda_{i}(v(t_{2}))|\geq\ell  L-\overline C\rho$\ .\\
\end{lemma}

\begin{proof}
The assumption that $v$ is valued in a connected component of $D$ is of course crucial.
The proof is a direct consequence of the fact that, in the region where $v$ is valued, the system is actually genuinely nonlinear with $\left|\nabla\lambda_{i}(u)\cdot r_{i}(u)\right|\geq\ell$.
Indeed, due to the choice of the parameterization jointly with Lemma \ref{L:acc2}, the Riemann problem having $v(t_{1})$, $v(t_{2})$ as left / right states, or viceversa, contains an $i$-rarefaction curve of strength at least $L-\dH \rho$ while the total strengths of other waves are less than $\dH \rho$.
Along that $i$ rarefaction curve $\lambda_{i}$ varies at least of $\ell  L$ since $\left|\nabla\lambda_{i}(u)\cdot r_{i}(u)\right|\geq\ell$, and then one hast the thesis by the Lipschitz continuity of $\lambda_{i}$.
\end{proof}

\subsubsection{Auxiliary lemmas}
\label{Ss:auxiliary}

\begin{lemma}
\label{L:distinctf}
Let $\varepsilon>0$.
Consider a space-like segment $P_{\nu}Q_{\nu}$ for which there are two distinct indexes $j< k$ such that each segment $P_{\nu}Q_{\nu}$ is crossed by an amount $\geq \varepsilon>0$ of both $j$-waves in $u_{\nu}$ and $k$-waves in $u_{\nu}$.
If $P_{\nu}\to \oP$ and $Q_{\nu}\to \oP$ then necessarily $\mu^{IC}(\{\oP\})>0$---see  Figure~\ref{fig:regioneDist}.
\end{lemma}
\begin{proof}
As the segment $P_{\nu}Q_{\nu}$ is space like, suppose for example that $P_{\nu}$ is on the left of $Q_{\nu}$.
Consider the region $\Gamma_{\nu}$ delimited by the leftmost $k$ forward and backward $k$-characteristics through $P_{\nu}$ and by the rightmost forward and backward $j$-characteristics through $Q_{\nu}$.
By strict hyperbolicity~\eqref{eq:strhyp} they intersect at points $R^{*}_{\nu}$, $S^{*}_{\nu}$ which converge to $\oP$ and therefore the region $\Gamma_{\nu}$ shrinks to the single point $\oP$.
By interaction estimates, there is an amount of interaction $\gtrsim \varepsilon^{2}$ in the region $\overline\Gamma_{\nu}$, which yields in the $\nu$-limit that $\mu^{IC}(\{\oP\})>0$.
\end{proof}

\begin{lemma}
\label{L:samef}
 Assume that the $i$-th family is piecewise genuinely nonlinear. Let $\varepsilon>0$.
If in some open region $\Gamma_{\nu}$ shrinking to a given point $\oP=(\ot,\ox)$, as $\nu\uparrow\infty$, in $u_{\nu}$ 
\begin{enumerate}
\item \label{a:1}
there is at most one $i$-jump $\og_{\nu}$ whose $j$-th component has strength more than $\varepsilon$ while the strengths of the $j$-th component of other $i$-waves are vanishingly small but 
\item  the total amount of strengths of the $j$-th component of $i$-waves different from $\og_{\nu}$
is more than $3\varepsilon$ and
\item\label{a:3} for all $k\neq i$ the amount of $k$-waves of $u_{\nu}$
is vanishingly small as $\nu\uparrow\infty$,
\end{enumerate}
then $\mu^{IC}(\{\oP\})>0$---see Figure~\ref{fig:regione1}.
\end{lemma}

\begin{proof}\firststep
\step{Finding two $i$-fronts $\gamma_{\nu}^{\ell}$, $\gamma_{\nu}^{r}$ of $u_{\nu}$ whose slopes at some time $t_{\nu}\to \ot$ remain distant without having any $(\beta,i,j)$-front in between}
In the hypothesis of the lemma, we can fix a space segment $P_{\nu}Q_{\nu}\subset\{t=t_{\nu}\}$, for $P_{\nu}Q_{\nu}\in\Gamma_{\nu}$, which is shrinking to the given point $\oP=(\ot,\ox)$ as $\nu\to0$, such that the total amount the $j$-th component of $i$-waves in $u_{\nu}$ along $P_{\nu}Q_{\nu}$ is more than $ \varepsilon$, but each one is vanishingly small as $\nu\uparrow\infty$.
Assume also that, if present, $\og_{\nu}$ lies, for example, on the right of $P_{\nu}Q_{\nu}$, $Q_{\nu}$ on the right of $P_{\nu}$ and
\[
u_{\nu}(P_{\nu})\to u^{L} \ ,
\qquad |u_{\nu}(Q_{\nu})-u^{L}|>\varepsilon 
\]
where $u^{L}$ denotes the left limit of $u(\ot,\cdot)$ at $\ox$, if suitably normalized.
We distinguish cases:
\begin{enumerate}

\item If $u^{L}\notin Z_{i}^{j}$ and $u^{L}\notin Z_{i}^{j+1}$ then Lemma~\ref{L:acc0} applies to $u_{\nu}$ restricted to an initial part of the interval $P_{\nu}Q_{\nu}$ and thus for some $c>0$ we can pick up a point $C_{\nu}$ belonging to the segment $P_{\nu}Q_{\nu}$ such that $\bullet$ as $\nu\uparrow\infty$
\[
|\lambda_{i}(u_{\nu}(C_{\nu})) - \lambda_{i}(u_{\nu}(P_{\nu}))|>c>0 \ ,
\]
such that $\bullet$ the total strength of $j$-components of $i$-fronts crossing $P_{\nu}C_{\nu}$ is at least $c$ and such that
$\bullet$ along the segment $P_{\nu}C_{\nu}$ the distance of $u$ from $Z_{i}^{j}$ and $ Z_{i}^{j+1}$ is at least $c$, so that $u$ takes values in a region where the $i$-characteristic field is genuinely nonlinear.

\item Even if $u^{L}\in Z_{i}^{j}$ or $u^{L}\in Z_{i}^{j+1}$, the maximum distance of $u_{\nu}(C)$ from $Z_{i}^{j}$ and $Z_{i}^{j+1}$, for $C$ varying in the interval $P_{\nu}Q_{\nu}$, cannot be vanishingly small: we could otherwise apply Lemma~\ref{L:acc2} to $u_{\nu}$ restricted on the whole segment $P_{\nu}Q_{\nu}$ and we would reach a contradiction with the assumption that the total amount of $i$-waves is not vanishingly small.
As the maximum distance of $u_{\nu}(C)$ from $  Z_{i}^{j}$ and $  Z_{i}^{j+1}$, for $C\in P_{\nu}Q_{\nu}$, is not vanishingly small, we can apply Lemma~\ref{L:acc0} to $u_{\nu}$ restricted on some sub-segment of $P_{\nu}Q_{\nu}$: for some $c>0$ we can thus pick up a point $C_{\nu}$ belonging to the segment $P_{\nu}Q_{\nu}$ such that $\bullet$ as $\nu\uparrow\infty$
\[
|\lambda_{i}(u_{\nu}(C_{\nu})) - \lambda_{i}(u_{\nu}(P_{\nu}))|>c>0 \ ,
\]
such that $\bullet$ the total strength of $j$-components of $i$-fronts crossing $P_{\nu}C_{\nu}$ is at least $c$ and such that
$\bullet$ along the segment $P_{\nu}C_{\nu}$ the distance of $u$ from $Z_{i}^{j}$ and $ Z_{i}^{j+1}$ is at least $c$, so that $u$ takes values in a region where the $i$-characteristic field is genuinely nonlinear.

\item If any limit point of $u_{\nu}(Q_{\nu})$, up to subsequence, belongs to the closed region between $ Z_{i}^{j}$ and $Z_{i}^{j+1}$ then the previous points still apply similarly.
In case not, by the assumptions one can replace $Q_{\nu}$ with another point $\overline{Q_{\nu}}$ having the same properties above and such that a limit point of the sequence $\overline{Q_{\nu}}$ falls in the closed region between $ Z_{i}^{j}$ and $Z_{i}^{j+1}$. 

\end{enumerate}
Consider now the leftmost $i$-discontinuity curve $\gamma_{\nu}^{\ell}$ of $u_{\nu}$, through $P_{\nu}$, and the rightmost one $\gamma_{\nu}^{r}$, through the point $C_{\nu}$ just determined.
Since the maximum size of jumps is vanishingly small we can also assume that
\ba\label{E:rfrgerqrg}
|\dot \gamma_{\nu}^{\ell}(t_{\nu})- \dot\gamma_{\nu}^{r}(t_{\nu})|>c>0
\ea
by the construction above of the point $C_{\nu}$.

\step{Conclusion when $\gamma_{\nu}^{\ell}$ and $\gamma_{\nu}^{r}$ meet at some time $\ot_{\nu}\to\ot$}
In case $\gamma_{\nu}^{\ell}$ and $ \gamma_{\nu}^{r}$ meet at time $\ot_{\nu}$ with
\[
|\ot_{\nu}-t_{\nu}|\leq \Delta_{\nu} 
\qquad
\text{where}
\qquad
\Delta_{\nu} :=4\frac{|P_{\nu}C_{\nu}|}{c} \ ,
\]
then denote by $ \Gamma_{\nu}$ the region delimited by the segment $P_{\nu}C_{\nu}$, by $\gamma_{\nu}^{\ell}$ and by $\gamma_{\nu}^{r}$ between times $\ot_{\nu}$ and $\ot $.
The region $\Gamma_{\nu}$ shrinks to $\oP$ and, by construction, {one can prove that} $\mu^{IC}_{\nu}\left(\overline \Gamma_{\nu}\right)$ is uniformly positive: this yields the thesis $\mu^{IC}(\oP)>0$.
Of course, $\gamma_{\nu}^{\ell}$ and $ \gamma_{\nu}^{r}$ necessarily meet by one such time $\ot_{\nu}$ if the slopes $\dot \gamma_{\nu}^{\ell}$, $\dot \gamma_{\nu}^{r}$ satisfy
\[
|\dot \gamma_{\nu}^{\ell}(q )- \dot\gamma_{\nu}^{r}(q )|>\frac{c}{4}
\qquad
\text{for all $q$ such that $|q-t_{\nu}|\leq\Delta_{\nu}  $.} 
\]

\step{Claim when $\gamma_{\nu}^{\ell}(q)$ and $\gamma_{\nu}^{r}(q)$ do not meet for $\left|q-t_{\nu}\right|\leq \Delta_{\nu}$}
In case $\gamma_{\nu}^{\ell}$ and $ \gamma_{\nu}^{r}$ do not meet in the time interval $\left[t_{\nu}-\Delta_{\nu},t_{\nu}+\Delta_{\nu}\right]$, we claim that
\ba
\label{E:gggggqgre}
\mu^{IC}_{\nu}\Big(i_{\gamma_{\nu}^{\ell}}\left(\left[t_{\nu}-\Delta_{\nu},t_{\nu}+\Delta_{\nu}\right]\right)\cup i_{\gamma_{\nu}^{r}}\left(\left[t_{\nu}-\Delta_{\nu},t_{\nu}+\Delta_{\nu}\right]\right)\Big)
\not\xrightarrow{} 0
\qquad \nu\uparrow\infty\ .
\ea
The symbol $i_{\gamma }$ denotes the map $s\mapsto (s,\gamma(s))$ for $s$ in the domain of $\gamma$.
We now prove~\eqref{E:gggggqgre}.
\step{Necessary condition if $\gamma_{\nu}^{\ell}(q)$ and $\gamma_{\nu}^{r}(q)$ do not meet of $\left|q-t_{\nu}\right|\leq \Delta_{\nu}$}
If $\gamma_{\nu}^{\ell}$ and $ \gamma_{\nu}^{r}$ do not meet for $\left|q-t_{\nu}\right|\leq\Delta_{\nu}$ then by~\eqref{E:rfrgerqrg} and by the triangular inequality necessarily:
\begin{enumerate}
\item In case $\gamma_{\nu}^{\ell}$ and $\gamma_{\nu}^{r}$ are approaching at time $t_{\nu}$---namely if
\[
\lambda_{i}(u_{\nu}(P_{\nu}))=\dot \gamma_{\nu}^{\ell}(\ot_{\nu}) >\dot \gamma_{\nu}^{r}(\ot_{\nu}) +c=\lambda_{i}(u_{\nu}(C_{\nu}))+c
\]
then at some $\ot_{\nu}$ with $t_{\nu}<\ot_{\nu}<t_{\nu}+\Delta_{\nu}$
\ba
\text{either }
\dot \gamma_{\nu}^{\ell}(\ot_{\nu}) <\dot \gamma_{\nu}^{\ell}( t_{\nu})-\frac{c}{4}
\qquad \text{ or }\qquad
 \dot \gamma_{\nu}^{r}(\ot_{\nu})>\dot \gamma_{\nu}^{r}( t_{\nu}) + \frac{c}{4} \ .
\ea
\item
If at time $t_{\nu}$ the fronts $\gamma_{\nu}^{\ell}$ and $\gamma_{\nu}^{r}$ are getting far apart---i.e.~$\dot \gamma_{\nu}^{\ell}(\ot_{\nu})<\dot \gamma_{\nu}^{r}(\ot_{\nu})$---then
\ba
\label{E:ggaagggr}
\text{either }
\dot \gamma_{\nu}^{\ell}(\ot_{\nu}) >\dot \gamma_{\nu}^{\ell}( t_{\nu})-\frac{c}{4}
 \qquad\text{ or } \qquad
 \dot \gamma_{\nu}^{r}(\ot_{\nu})<\dot \gamma_{\nu}^{r}( t_{\nu}) + \frac{c}{4}  
 \ea
at some $\ot_{\nu}$ with $t_{\nu}-\Delta_{\nu}<\ot_{\nu}<t_{\nu} $.
\end{enumerate}
\step{Proof of the claim} Suppose e.g~$\gamma_{\nu}^{\ell}(q)$ and $\gamma_{\nu}^{r}(q)$ do not meet for $t_{\nu}-\Delta_{\nu}<q<t_{\nu} $, as in the other case the analysis is analogous.
Let $\ot_{\nu}$ be as in~\eqref{E:ggaagggr}.
We first estimate how the slope of an $i$-front $\gamma$ might vary.
Collecting the estimates in Lemma~\ref{L:timeEst} at each update time in $[\ot_{\nu},t_{\nu}]$ jointly with interaction estimates in~\cite[Lemma~1]{AMfr}, we find the rough estimate
\ba
\label{E:gerqgrrgg}
\left|\dot \gamma (\ot_{\nu})
-\dot \gamma (t_{\nu})
\right|
\lesssim 
\Delta_{\nu}+\tau_{\nu}
+M_{ i}^{\nu}(\gamma,\ot_{\nu},t_{\nu})+M_{ *}^{\nu}(\gamma,\ot_{\nu},t_{\nu}) + \mu^{IC}_{\nu}\left(i_{\gamma}([\ot_{\nu},t_{\nu}])\right) 
\ea
where we adopted the following notation:
\begin{itemize}
\item $M_{*}^{\nu}(\gamma,r,t)$ is the strength of all $k$-waves, $k\neq i$, interacting with $\gamma $ between times $r$ and $t$;
\item $M_{  i}^{\nu}(\gamma,r,t)$ is the strength of $i$-waves interacting with $\gamma $ between times $r$ and $t$.

\end{itemize}

By assumption $M_{*}^{\nu}(\gamma_{\nu}^{\ell},\ot_{\nu},t_{\nu})$, $M_{*}^{\nu}(\gamma_{\nu}^{r},\ot_{\nu},t_{\nu})$, $\Delta_{\nu}$, $\tau_{\nu}$ are vanishingly small as $\nu\uparrow\infty$.
If also $M_{ i}^{\nu}(\gamma_{\nu}^{\ell},\ot_{\nu},t_{\nu})$ and $M_{ i}^{\nu}(\gamma_{\nu}^{r},\ot_{\nu},t_{\nu})$ are vanishingly small, then~\eqref{E:gerqgrrgg} directly implies~\eqref{E:gggggqgre} for $\nu$ large enough, so that $\mu^{IC}(\oP)>0$.

Suppose instead that one among $M_{ i}^{\nu}(\gamma_{\nu}^{\ell},\ot_{\nu},t_{\nu})$ and $M_{ i}^{\nu}(\gamma_{\nu}^{r},\ot_{\nu},t_{\nu})$ is not vanishingly small.
Denote simply by $\gamma$ such front.
Assume also that $j$ is even, so that $(\beta,i,j)$-sub-discontinuity fronts correspond to positive parameters $s$ in Definition~\ref{D:subdisc1}, for notational convenience.

By Lemma~\ref{L:rqgsbb} applied with $\Gamma$ a neighborhood of $i_{\gamma}([\ot_{\nu},t_{\nu}])$, one has
\ba
\label{E:gtoooiioi}
 |(s_{i} )^{+}(t_{\nu}+)-(s_{i} )^{+}(\ot_{\nu}-)-M_{ i}^{+\nu}(\gamma, \ot_{\nu},t_{\nu})|\lesssim \Delta_{\nu}+\tau_{\nu}+\mu^{IC}_{\nu}\left(i_{\gamma}([\ot_{\nu},t_{\nu}])\right)  \\
 \label{E:tejknjrhjrg}
 |(s_{i} )^{-}(t_{\nu}+)-(s_{i} )^{-}(\ot_{\nu}-)-M_{ i}^{-\nu}(\gamma, \ot_{\nu},t_{\nu})|\lesssim \Delta_{\nu}+\tau_{\nu}+\mu^{IC}_{\nu}\left(i_{\gamma}([\ot_{\nu},t_{\nu}])\right)  
\ea
where $|s_{i}|$ is the strength of the $i$-front $\gamma$ and the apex ${}^{+}$ (resp.~${}^{-}$) means that we are taking into account only strengths of positive (resp.~negative) $i$-waves.
Since by construction 
\[
(s_{i})^{+}(t_{\nu}+)=s_{i}^{j}(t_{\nu}+)\ ,
\qquad 
(s_{i})^{-}(t_{\nu}+) \ ,
\qquad
(s_{i})^{-}(\ot_{\nu}-)
\] 
are vanishingly small---the first one by the choice of the point $C_{\nu}$ / $P_{\nu}$ while the second one since single rarefaction fronts have vanishingly small strengths---and since $ (s_{i} )^{+}(\ot_{\nu}-)$, $M_{ i}^{+\nu}(\gamma, \ot_{\nu},t_{\nu})$ are both nonnegative by definition, then
\begin{itemize}
\item either $ (s_{i} )^{+}(\ot_{\nu}-)+M_{ i}^{+\nu}(\gamma, \ot_{\nu},t_{\nu})$ is not vanishingly small, so that \eqref{E:gtoooiioi} implies~\eqref{E:gggggqgre},
\item or $ M_{ i}^{-\nu}(\gamma, \ot_{\nu},t_{\nu})$ is not vanishingly small, from which~\eqref{E:tejknjrhjrg} implies~\eqref{E:gggggqgre},
\end{itemize}
or both of the cases happen.
This concludes the proof of the theorem.
\end{proof}
\begin{remark}
Lemma~\ref{L:samef} extends with little modification to the case of two or more
$i$-jumps converging to $\oP$ whose $j$-th component has strength more than
$\varepsilon$, rather than a single one, both with or without any uniformly positive
amount of $i$-waves of vanishingly small strength.
\end{remark}

%
%
%
%
\section{Essential nomenclature}

$\mathrm{ceil}$: Smallest integer bigger than a given real number, i.e.~integer part of the number plus one. 
\\[.4\baselineskip]
$ \delta_{hk}$: Delta di Kronecker, which is equal to $1$ if $h=k$ and it vanishes otherwise.
\\[.4\baselineskip]
$\hat\lambda$: Uniform bound for the characteristic speeds of very family.
\\[.4\baselineskip]
$\vSC{t}{h}\bar u$: The viscous semigroup of the Cauchy problem~\eqref{eq:sysnc}-\eqref{eq:inda} starting at time $h$, rather than fixing the initial time $h=0$. See~\cite{Ch1}.
\\[.4\baselineskip]
$\vSB{t}{h}$: The semigroup of the Cauchy problem for the homogeneous system~\eqref{eq:sysnc}-\eqref{eq:inda} when $g\equiv 0$ starting at time $h$, rather than fixing the initial time $h=0$. See~\cite{BB}.
\\[.4\baselineskip]
$\ftS{t}{h}$: The $\varepsilon_{\nu}$-wave-front tracking approximation of $\vSB{t}{h}$ by~\cite{AMfr}.
\\[.4\baselineskip]
$\Omega$: Open, bounded, connected subset of $\real^N$ where $u$ is valued.
\\[.4\baselineskip]
$\dssb$: Smallness parameter for initial datum in the Cauchy problem of the viscous system~\eqref{eq:sysnc} as in Theorem~\ref{T:localConv} \\[.4\baselineskip]
$ \dH$: Smallness parameter for initial datum in the Cauchy problem of the homogeneous system as at Page~\pageref{E:rgabab} which provides a threshold for global existence and convergence of wave-front tracking approximations. See~\cite{AMfr}, where it is denoted by $\delta_0$.
\\[.4\baselineskip]
$\varepsilon$: Positive parameter.
\\[.4\baselineskip]
$\nu$: Positive integer parameter relative to subsequences of $\varepsilon_{\nu}$-wave-front tracking approximations or of $(\varepsilon_{\nu},\tau_{\nu})$-fractional-step approximations. See~\S~\ref{sec:PCA}.
\\[.4\baselineskip]
$\varepsilon_{\nu}$: Positive vanishing constant, as $\nu\uparrow\infty$, in $\varepsilon_{\nu}$-wave-front tracking approximation and in $(\varepsilon_{\nu},\tau_{\nu})$-fractional-step approximation. It correspond also to maximum size of rarefactions.
\\[.4\baselineskip]
$\tau_{\nu}$: Size of the time-step in $(\varepsilon_{\nu},\tau_{\nu})$-fractional-step approximations. See~\S~\ref{sec:PCA}.
\\[.4\baselineskip]
$u_{\nu}$, $u$: Often, $\varepsilon_{\nu}$-wave front tracking approximation of the Cauchy problem~\eqref{eq:sysnc}-\eqref{eq:inda}, either homogeneous or not, and its limit entropy solution, as constructed in~\cite{AMfr} and recalled in \S~\ref{Ss:frontTr}.
\\[.4\baselineskip]
$\ww_{\nu}$, $\ww$: Fractional step approximation, and its limit as $\nu\uparrow\infty$, of the Cauchy problem~\eqref{eq:sysnc}-\eqref{eq:inda} that we construct in~\S\S~\ref{sec:PCA}-\ref{S:convergenceAndExistence}. We fix the right-continuous representative in time and space.
\\[.4\baselineskip]
$\ell_{g}$: The Lipschitz constant of $g(\cdot, x,u)$ in $x,u$. See the assumption (G) at Page~\pageref{Ass:G}.
\\[.4\baselineskip]
$\alpha$: See the assumption (G) at Page~\pageref{alpha}.
\\[.4\baselineskip]
$\Phi$: Functional defined at~\eqref{eq2:RP1}.
\\[.4\baselineskip]
$\Q$, $\mathcal{V}$, $\Upsilon$: Functionals defined
in~\eqref{eq:ip}-\eqref{eq:Ups}.
\\[.4\baselineskip]
$\lesssim$ Less or equal up to a constant which only depends only on the flux of~\eqref{eq:sysnc} and on $\dss$ or $\dH$.
\\[.4\baselineskip]
$\mu_{(\nu)}^{I}$, $\mu_{(\nu)}^{IC}$: Interaction and interaction-cancellation measures, either on a sequence of ($\varepsilon_{\nu}$-$\tau_{\nu}$)-fractional step approximations or a fixed limit of them. See~\eqref{E:muIC}.
\\[.4\baselineskip]
$(\beta,i,k)$-approximate sub-discontinuity as defined in \S~\ref{S:discontinuities}.
\\[.4\baselineskip]
$\mu^{i*}_{(\nu)}, \overline{\mu}_{(\nu)}^{i\pm}$: Real measures detecting possible initial points of $i$-contact discontinuities, \S~\ref{Ss:constructionlimitcurves}.
\\[.4\baselineskip]
$\J_{i}$ $\J_{\beta,i} $ $\J$: $i$-Shock-fronts and $i$-contact discontinuity fronts,e \S~\ref{Ss:constructionlimitcurves}.
\\[.4\baselineskip]
$\Theta$: At most countable set containing interaction points, \S~\ref{Ss:constructionlimitcurves}.
\\[.4\baselineskip]
$\TV(v)$: The total variation of $v:\real\to\real^m$, $m\in\nat$, which is $\sup_{K\in\nat}\sup_{x_{1}<\dots<x_{K}}\sum_{j}|v(x_{j+1})-v(x_{j})|$. 
\\[.4\baselineskip]
\bibliographystyle{plain}

\end{document}